\documentclass[bj, preprint]{imsart}

\usepackage{amsthm,amsmath,amsfonts,xcolor}
\usepackage[square,sort,comma,numbers]{natbib}
\startlocaldefs
\newtheorem{theorem}{Theorem}
\newtheorem{prop}{Proposition}
\newtheorem{lemma}{Lemma}
\theoremstyle{definition}
\newtheorem{definition}{Definition}
\theoremstyle{remark}
\newtheorem{remark}{Remark}
\newtheorem{example}{Example}
\newcommand{\E}{\mathbb{E}} %for expectation
\newcommand{\R}{{\mathbb R}}
\newcommand{\N}{{\mathbb N}}
\endlocaldefs

\begin{document}

	\begin{frontmatter}
		
		% "Title of the Paper"
		\title{The unusual properties of aggregated superpositions of Ornstein-Uhlenbeck type processes}
		
		\runtitle{The unusual properties of supOU processes}
		
		\begin{aug}
			% indicate corresponding author with \corref{}
			 \author{\fnms{Danijel} \snm{Grahovac} \thanksref{a}\corref{}\ead[label=e1]{dgrahova@mathos.hr}}%\ead[label=e2,url]{www.foo.com}}
			 \address[a]{Department of Mathematics, University of Osijek, Trg Ljudevita Gaja 6, 31000 Osijek, Croatia. \printead{e1}}
			
			\author{\fnms{Nikolai N.} \snm{Leonenko}\thanksref{b}\ead[label=e2]{LeonenkoN@cardiff.ac.uk}}
			
			\address[b]{School of Mathematics, Cardiff University, Senghennydd Road, Cardiff, Wales, UK, CF24 4AG. \printead{e2}}
			
			\author{\fnms{Alla} \snm{Sikorskii}\thanksref{c}\ead[label=e3]{sikorska@stt.msu.edu}}
			
			\address[c]{Department of Psychiatry and Department of Statistics and Probability, Michigan State University, East Lansing, MI 48824, USA. \printead{e3}}
			\and
			\author{\fnms{Murad S.} \snm{Taqqu}\thanksref{d}				\ead[label=e4]{murad@bu.edu}}%
						
			\address[d]{Department of Mathematics and Statistics, Boston University, Boston, MA 02215, USA. \printead{e2}}
			
			\runauthor{D.~Grahovac et al.}
			
			\affiliation{University of Osijek}
			
		\end{aug}
		
		\begin{abstract}
			Superpositions of Ornstein-Uhlenbeck type (supOU) processes form a rich class of stationary processes with a flexible dependence structure. The asymptotic behavior of the integrated and partial sum supOU processes can be, however, unusual. Their cumulants and moments turn out to have an unexpected rate of growth. We identify the property of fast growth of moments or cumulants as \textit{intermittency}. 
%			and show how it relates to the asymptotic behavior of the process.
	\end{abstract}
		
		
		% history:
		% \received{\smonth{1} \syear{0000}}
		
		%\tableofcontents
		
	\end{frontmatter}

\bigskip
\large{\today}\\
\bigskip

\section{Introduction}
L\'evy driven Ornstein-Uhlenbeck (OU) processes form a rich class of stationary processes with mixing properties. They can have any selfdecomposable distribution as their marginal distribution. Superpositions of OU type (supOU) processes  were introduced by Barndorff-Nielsen in \cite{barndorff1997processes} and \cite{bn2001} using a construction that was later generalized to obtain L\'evy mixing processes (see \cite{barndorff2013levy}). The supOU processes are stationary processes with a flexible dependence structure. A square integrable stationary process $X(t)$, $t\ge 0$ is said to have {\it short-range dependence} if its correlation function is integrable and {\it long-range dependence} if it is not integrable. It is possible for supOU processes to display not only short-range dependence but also long-range dependence. SupOU processes have found many applications, especially in finance where positive supOU processes are used in models for stochastic volatility; see \cite{bnshepard2001,moser2011tail,barndorff2013multivariate,barndorff2013stochastic,Stelzer2015,stelzer2015moment,griffin2010bayesian}.

In this paper we discuss the asymptotic properties of two variants of aggregated supOU process: \textit{the integrated process} obtained from a continuously observed supOU process and \textit{the partial sum process} obtained from a discretely sampled supOU process. These are of particular interest in finance where the integrated process represents the integrated volatility (see e.g.~\cite{barndorff2013multivariate}). When there are only finitely many OU type processes in the superposition, the mixing property remains valid and implies the convergence of the aggregated process to Brownian motion (see \cite{GLST2015}). Problems arise when one considers an infinite superposition of OU type processes. This paper provides a closer analysis to the corresponding behavior of moments and cumulants. Several attempts have been made to associate that behavior to rates in limit theorems but to no avail, see for example \cite{barndorff2005burgers,leonenko2005convergence}.

Intermittency, which will be defined bellow, refers to this unusual behavior of moments and cumulants. Note that our definition of intermittency will be different from the one used in \cite{barndorff2009brownian,barndorff2014assessing,podolskij2015}, where intermittency is associated with stochastic volatility. Here, as in the physics literature, intermittency is associated with the behavior of moments (\cite{zel1987intermittency,carmona1994parabolic}).

In order to study the asymptotic behavior of the aggregated processes, we investigate how the cumulants and moments evolve in time. The classical limiting scheme for some type of aggregated process $Y=\{Y(t),\, t \geq 0\}$ has the form
\begin{equation}
    \left\{ \frac{Y(nt)}{A_n} \right\} \overset{d}{\to} \left\{ Z(t) \right\}, \label{limitform}
\end{equation}
with convergence in the sense of convergence of all finite dimensional distributions as $n \to \infty$. By Lamperti's theorem (see, for example, \cite[Theorem 2.1.1]{embrechts2002selfsimilar}), the normalizing sequence is always of the form $A_n=L(n) n^H$ for some $H>0$ and $L$ slowly varying at infinity. Moreover, the limiting process $Z$ is $H$-self-similar, that is, for any $c>0$, $\{Z(ct)\}\overset{d}{=} \{c^H Z(t)\}$, where $\{\cdot\} \overset{d}{=} \{\cdot\}$ denotes the equality of finite dimensional distributions. For self-similar process, the moments evolve as a power function of time $\E|Z(t)|^q=\E|Z(1)|^q t^{Hq}$. Hence, for the process $Y$ satisfying a limit theorem in the form \eqref{limitform}, one expects that 
\begin{equation}
\frac{\E| Y(nt)|^q}{A_n^q} \to \E |Z(t)|^q, \quad \forall t \geq 0. \label{limitformmom}
\end{equation}
Therefore, $\E|Y(t)|^q$ grows roughly as $t^{Hq}$ when $t\to \infty$.
Indeed, ignoring the slowly-varying function $L$ and multiplicative constants, we have
$$
\E| Y(nt)|^q \approx n^{Hq}  \E |Z(t)|^q \approx n^{Hq} t^{Hq} \E |Z(1)|^q \approx (nt)^{Hq},
$$
 and hence
\begin{equation} \label{e:contradict1}
\E| Y(t)|^q \approx t^{Hq}\ \mbox{\rm as}\ t \rightarrow \infty.
\end{equation}
 (see Theorem \ref{theorem:limit} below for the precise statement).
 
We study aggregated processes $Y(t)$
 arising from supOU processes with a regularly varying correlation function and a marginal distribution having exponentially decaying tails, so that, in particular, all moments are finite. We show that these aggregated processes have a specific growth of moments: for a certain range of $q$, 
namely 
\begin{equation} \label{e:contr}
\E|Y(t)|^q \approx  t^{q-\alpha}\ \mbox{\rm as}\ t\to \infty,
\end{equation}
 \textit{Relation \eqref{e:contr}
 	contradicts (\ref{e:contradict1}).}
 Here $\alpha$ is the parameter related to the dependence structure of the underlying supOU process (see Theorems \ref{cor:intermittency-integrated} and \ref{cor:intermittency-partilsum} bellow).

We show that in our context the growth of the cumulants and moments is such that the relation between \eqref{limitform} and \eqref{limitformmom} falls apart. We refer to this property as \textit{intermittency}. The term is usually used to describe models exhibiting a high degree of variability and appears in different contexts across the literature; see e.g.~\cite{carmona1994parabolic,frisch1995turbulence,gartner2007geometric,khoshnevisan2014analysis,zel1987intermittency,chen2015moments}. Inspired by these approaches, we define intermittency as a property arising from a particular growth of moments. A precise definition is given in Section \ref{sec2}. 
In that section,
we show that for intermittent processes either a limit theorem as in \eqref{limitform} and convergence of moments \eqref{limitformmom} do not work together (see Theorem \ref{theorem:limit} below).
 %The paper is organized as follows.
 
 Section \ref{sec3} provides an overview of facts relevant for the definition and properties of supOU processes. The expressions for cumulants are established for aggregated processes. In Section \ref{sec4}, the growth of cumulants is analyzed 
and we show in Theorems \ref{cor:intermittency-integrated} and
\ref{cor:intermittency-partilsum}
respectively that the integrated process and the partial sum of supOU processes can be intermittent.

\section{Intermittency}\label{sec2}
Intermittency is a property used to describe models exhibiting sharp fluctuations in time and a high degree of variability. Terms such as multifractality, separation of scales, dynamo effect are often used together with intermittency. The term has a precise definition in the theory of stochastic partial differential equations (SPDE), where it is characterized by the Lyapunov exponents (see e.g.~\cite{zel1987intermittency,carmona1994parabolic,khoshnevisan2014analysis,chen2015moments}). The $k$-th moment Lyapunov exponent of a non-negative random field $\{\psi(t,x), t\geq 0, x \in \R \}$ stationary in $x$ is defined by
\begin{equation}\label{SPDE_ver}
    \gamma(k)=\lim_{t\to\infty} \frac{\log \E\left(\psi(t,x)\right)^k}{t},
\end{equation}
assuming the limit exists and is finite. A random field $\{\psi(t,x)\}$ is then said to be intermittent if the sequence $\gamma(k)/k$, $k\in \N$ is strictly increasing, that is
\begin{equation*}
  \gamma(1) < \frac{\gamma(2)}{2} < \cdots < \frac{\gamma(k)}{k} < \cdots.
\end{equation*}
This property can be shown to imply under some assumptions that the random field has large peaks at different values of the space coordinate (see \cite{molchanov1991ideas,khoshnevisan2014analysis} for details).

We define \textit{intermittency} as a property which indicates that the moments of the stochastic process do not have a typical limiting behavior. Our focus will be on the behavior of the moments of the process in time as characterized by the scaling function defined below. The Lyapunov exponents are suitable for measuring the growth rate of random fields that have moments that grow exponentially in time. On the other hand, the scaling function is tailored for cumulative processes, e.g.~partial sum process, whose limiting behavior is investigated.

For a process $Y=\{Y(t),\, t \geq 0\}$, let $(0,\overline{q}(Y))$ denote the range of finite moments, that is
\begin{equation*}
    \overline{q}(Y) = \sup \{ q >0 :\E|Y(t)|^q < \infty  \ \forall t\}.
\end{equation*}
\begin{definition}
The \textbf{scaling function} at point $q \in (0,\overline{q}(Y))$ of the process $Y$ is
\begin{equation}\label{deftau}
    \tau_Y(q) = \lim_{t\to \infty} \frac{\log \E |Y(t)|^q}{\log t},
\end{equation}
assuming the limit exists and is finite.
\end{definition}
Note the difference between \eqref{SPDE_ver} and \eqref{deftau}. In our context, it is the scaling function \eqref{deftau} which is relevant.
It  can be shown that $\tau_Y$ is always convex and $q \mapsto \tau_Y(q)/q$ is non-decreasing (\cite{GLST2015}). Using the scaling function we characterize intermittency as a strict increase in the mapping $q \mapsto \tau_Y(q)/q$.

\begin{definition}
A stochastic process $Y=\{Y(t),\, t \geq 0\}$ is \textbf{intermittent} if there exist $p, r \in (0,\overline{q}(Y))$ such that
\begin{equation}\label{intermittency}
    \frac{\tau_Y(p)}{p} < \frac{\tau_Y(r)}{r}.
\end{equation}
\end{definition}

If $Y$ is a $H$-self-similar process, then $\tau_Y(q)=Hq$, and $\tau_Y(q)/q$ is constant, therefore the process is not intermittent. The following theorem shows that when the process $Y$ is not self-similar but has a {\it typical limit behavior} as described in the theorem (in particular, convergence to a self-similar process after suitable normalization) and if the corresponding moments converge, then its scaling function $\tau_Y$ turns out to be the same as for the self-similar process, namely $\tau_Y(q)=Hq$ for some $H>0$. 
%This happens, for example, if $Y$ is an integral or a partial sum of a stationary process subordinated to the Gaussian process with short or long-range dependence. However, when intermittency is present, the following typical limit behavior does not hold.

\begin{theorem}\label{theorem:limit}
Let $Y=\{Y(t), \, t \geq 0\}$ and $Z=\{Z(t), \, t \geq 0\}$ be two processes such that $Z(t)$ is nondegenerate for every $t>0$ and suppose that for a sequence $(A_n)$, $A_n>0$, $\lim_{n\to \infty} A_n = \infty$, one has
\begin{equation}
    \left\{ \frac{Y(nt)}{A_n} \right\} \overset{d}{\to} \left\{ Z(t) \right\}, \label{thm:limitdistr}
\end{equation}
with convergence in \eqref{thm:limitdistr} in the sense of convergence of all finite dimensional distributions as $n \to \infty$. Then there exists a constant $H>0$ such that for every $q>0$ satisfying
\begin{equation}
    \frac{\E| Y(nt)|^q}{A_n^q} \to \E |Z(t)|^q, \quad \forall t \geq 0, \label{thm:limitmom}
\end{equation}
the scaling function \eqref{deftau} of $Y$ at $q$ is
\begin{equation}\label{e:Hq}
    \tau_Y(q)= H q.
\end{equation}
\end{theorem}

\begin{proof}
By Lamperti's theorem (see, for example, \cite[Theorem 2.1.1]{embrechts2002selfsimilar}), \eqref{thm:limitdistr} implies the process $Z$ is $H$-self-similar with $H>0$ and $A_n$ is of the form $A_n=n^H L(n)$ for some function $L$ slowly varying at infinity. It follows from \eqref{thm:limitmom} that
\begin{eqnarray*}
  \log \frac{\E| Y(nt)|^q}{A_n^q} &=&
  \log \E| Y(nt)|^q - \log  (n^H L(n))^q \\
   &=& \log n \left( \frac{\log \E| Y(nt)|^q }{\log nt} \frac{\log nt}{\log n} - q\frac{\log \left( n^H L(n) \right)}{\log n} \right) \\
   &\to& \log \E| Z(t)|^q \text{ as } n \to \infty.
\end{eqnarray*}

Thus the factor in the parentheses that multiplies $\log n$ in the above equation must tend to zero as $n\to \infty $. Since $\log nt /\log n \to 1 $ as $n\to \infty $,
by \cite[Proposition 1.3.6(i)]{bingham1989regular}
\begin{equation*}
\lim _{n\to \infty }   \frac{\log \E| Y(nt)|^q }{\log nt} =\lim _{n\to \infty }  q\frac{\log \left( n^H L(n) \right)}{\log n} = H q + \lim _{n\to \infty } q\frac{\log L(n)}{\log n}=Hq.
\end{equation*}
Hence
$
\tau_Y(q)= H q.
$
\end{proof}

\begin{remark}
Assumption \eqref{thm:limitdistr} is the typical form in which limit theorems appear with $Y$ being a partial sum process or an integrated process. The limiting process is always self-similar, and the normalizing sequence is regularly varying. If in addition to \eqref{thm:limitdistr} convergence of moments holds, then $Y$ has a linear scaling function \eqref{e:Hq} and is not intermittent. \textit{Therefore, in the intermittent case either \eqref{thm:limitdistr} or \eqref{thm:limitmom} or both must fail to hold.}
\end{remark}

\begin{remark}
Notice that the scaling function involves only the one-dimensional marginal distributions of the process. Moreover, the conclusion of Theorem \ref{theorem:limit} holds if we assume that convergence in \eqref{thm:limitdistr} holds only for one-dimensional marginals. Indeed, from the proof of Lamperti's theorem \cite[Theorem 2.1.1]{embrechts2002selfsimilar}) this is enough to imply that $A_n=n^H L(n)$, and the same argument as in the proof of Theorem \ref{theorem:limit} applies.
\end{remark}

\begin{remark}
The relation between \eqref{thm:limitdistr} and \eqref{thm:limitmom} is a well known problem. In one direction, for a sequence of random variables convergence of moments implies weak convergence if the limiting distribution is uniquely determined by its moments. The question whether this is true is known as the moment problem (see e.g.~\cite[Section 11.]{stoyanov1997counterexamples} and references therein). On the other hand, for a sequence of random variables convergence of moments is implied by the weak convergence if the appropriately transformed sequence is uniformly integrable.
\end{remark}

Depending on the problem considered, it may be easier to establish intermittency by considering cumulants instead of moments. For $m\in \N$ and $t \geq 0$, let $\kappa_Y^{(m)}(t)$ denote the $m$-th order cumulant of $Y(t)$. The corresponding cumulant variant of the scaling function can be defined as
\begin{equation}\label{defsigma}
  \sigma_Y(m) = \lim_{t\to \infty} \frac{\log \left| \kappa_Y^{(m)}(t) \right|}{\log t}, \quad m \in \N,
\end{equation}
assuming $\kappa_Y^{(m)}(t)\neq 0$ and the limit exists and is finite. When the form of $\sigma_Y$ is established, the relation between moments and cumulants can be used to obtain the expression for $\tau_Y$. Note, however, that both
 (\ref{deftau})
and (\ref{defsigma}) involve absolute values.

In the next section, we review basic facts about the supOU processes. These provide great flexibility in modeling of stationary phenomena. This is becuse a supOU process can be chosen to have any selfdecomposable distribution as its marginal distribution and a variety of correlation structures. Some particular choices will lead to intermittent cumulative processes.

\section{SupOU processes}\label{sec3}
In order to define superpositions of OU type processes we introduce some notation and review basic facts about random measures and OU type processes.

\subsection{Preliminaries}
Let
\begin{equation*}
\kappa_Y(\zeta)=C\left\{ \zeta \ddagger Y\right\} = \log \E e^{i \zeta Y}
\end{equation*}
denote the cumulant (generating) function of a random variable $Y$ and, assuming it exists, $\kappa_Y^{(m)}$ for $m \in \N$ will denote the $m$-th cumulant of $Y$, that is
\begin{equation*}
  \kappa_Y^{(m)} = (-i)^m \frac{d^m}{d\zeta^m} \kappa_Y(\zeta) \big|_{\zeta=0}.
\end{equation*}
If $\kappa _Y(\cdot)$ is analytic around the origin, then
\begin{equation}\label{e:cgf}
\kappa_Y(\zeta)=\sum _{m=1}^\infty \frac {(i\zeta )^m}{m!}\kappa_Y^{(m)}.
\end{equation}
For a stochastic process $Y=\{Y(t)\}$ we write $\kappa_Y(\zeta,t) = \kappa_{Y(t)}(\zeta)$, and by suppressing $t$ we mean $\kappa_Y(\zeta)=\kappa_Y(\zeta,1)$, that is the cumulant function of the random variable $Y(1)$. Similarly, for the cumulants of $Y(t)$, we use the notation $\kappa_Y^{(m)}(t)$ and $\kappa_Y^{(m)}$ for $\kappa_Y^{(m)}(1)$. Recall that the cumulant function of infinitely divisible random variable $Y$ has the L\'{e}vy-Khintchine representation
\begin{equation*}
C\left\{ \zeta \ddagger Y\right\} =ia\zeta -\frac{b}{2}\zeta^{2}+\int_{\R}\left( e^{i\zeta x}-1-i\zeta \mathbf{1}_{[-1,1]}(x)\right) \mu(dx), \quad \zeta \in \R
\end{equation*}
where $a\in \R$, $b>0$, and the \textit{L\'{e}vy measure} $\mu$ is a deterministic Radon measure on $\R\backslash \{0\}$ such that $\mu\left( \left\{ 0\right\} \right) =0$ and $\int_{\R} \min \left\{ 1,x^{2}\right\} \mu(dx)<\infty$. The triplet $(a,b,\mu)$ is referred to as the \textit{characteristic triplet}. A stochastic process $\{L(t), \, t\geq 0\}$ with stationary, independent increments and continuous in probability ($L(t) \to^P 0$ as $t\to 0$) has a c\`adl\`ag modification which we refer to as a \textit{L\'evy process}. For any infinitely divisible random variable $Y$, there is a corresponding L\'evy process $\{L(t), \, t \geq 0\}$ such that $Y =^d L(1)$.

An infinitely divisible random variable $X$ is \textit{selfdecomposable} if its characteristic function $\phi (\theta)=\E e^{i\theta X}$, $\theta \in \R$, has the property that for every $c\in (0,1)$ there exists a characteristic function $\phi_{c}$ such that $\phi (\theta)=\phi (c\theta)\phi _{c}(\theta)$ for all $\theta\in \R$.
This means that that $X$ has the same distribution as $cX+Z_c$, where $X$ and $Z_c$ and independent, and $Z_c$ has the characteristic function $\phi _c$.
In this case, $X$ can be represented as
\begin{equation}\label{sdrepresentation}
  X= \int_0^{\infty }e^{-s} dL(s),
\end{equation}
where $L=\{L(t), \, t\geq 0\}$ is a L\'{e}vy process whose law is determined uniquely by that of $X$. The process $L$ is called the \textit{background driving L\'{e}vy process} (BDLP) corresponding to the infinitely divisible random variable $X$. The cumulant functions of $X$ and $L(1)$ are related by
\begin{equation}\label{kappaXtoL}
\kappa_{X}(\zeta )=\int_{0}^{\infty }\kappa_{L}(e^{-s}\zeta) ds.
\end{equation}
From \cite[Corollary 1]{jurek2001remarks} $\kappa_{X}$ is differentiable for $\zeta \neq 0$, $\zeta \kappa_X'(\zeta ) \to 0$ as $0\neq \zeta \to 0$ and
\begin{equation}\label{kappaLtoX}
\kappa_{L}(\zeta )=\zeta \kappa_X'(\zeta ).
\end{equation}
The BDLP $L$ can be extended to a two-sided L\'evy process by putting for $t<0$, $L(t)=-\widetilde{L}(-t-)$ where $\{\widetilde{L}(t), \, t \geq 0\}$ is an independent copy of the process $\{ L(t), \, t \geq 0\}$ modified to be c\`adl\`ag. The \textit{Ornstein-Uhlenbeck type (OU) process} is a process $\{X(t), \, t\in \R\}$ defined by
\begin{equation}\label{xonRdef}
  X(t) = e^{-\lambda t} \int_{-\infty}^t e^{\lambda s} dL (\lambda s) = \int_{\R} e^{-\lambda t + s} \mathbf{1}_{[0,\infty)}(\lambda t-s) dL (s),
\end{equation}
where $\lambda>0$.
It can be shown that $\{X(t), \, t\in \R\}$ is strictly stationary with the stationary distribution equal to the selfdecomposable law of $X$ corresponding to the BDLP $L$. When $X(t)$ has a finite second moment, the correlation function is $r(\tau)=e^{-\lambda \tau}$, $\tau\geq 0$ (\cite{bn2001}). Alternatively, starting with a L\'evy process $L$ satisfying $\E \log \left(1+ \left| L(1) \right| \right)< \infty$, one can define an OU type process as a stationary solution of the stochastic differential equation
\begin{equation*}%\label{sde}
dX(t)=-\lambda X(t)dt + dL(\lambda t).
\end{equation*}

We now turn to supOU processes. To define them, we need some basic facts about \textit{infinitely divisible independently scattered random measures} (i.d.i.s.r.m.). Let $S$ be a Borel subset of $\R^{d}$ and let $\mathcal{S}$ be a $\sigma$-ring of $S$ (i.e.~countable unions of sets in $\mathcal{S}$ belong to $\mathcal{S}$ and if $A$ and $B$ are sets in $\mathcal{S}$ with $A\subset B$, then $B\backslash A \in \mathcal{S}$). A collection of random variables $\Lambda=\left\{ \Lambda(A), A\in \mathcal{S} \right\}$ defined on a probability space $(\Omega,\mathcal{F},P)$ is said to be an \textit{independently scattered random measure} if for every sequence $\left\{ A_{n}\right\} $ of disjoint sets in $\mathcal{S}$, the random variables $\Lambda(A_{n})$, $n=1,2,...$ are independent and if
\begin{equation*}
\Lambda\left( \bigcup\limits_{n=1}^{\infty }A_{n}\right) =\sum_{n=1}^{\infty} \Lambda(A_{n}) \quad a.s.
\end{equation*}%
whenever $\bigcup_{n=1}^{\infty}A_{n}\in \mathcal{S}$. We will be interested in the case when $\Lambda$ is infinitely divisible, that is, for each $A\in \mathcal{S}$, $\Lambda(A)$ is an infinitely divisible random variable whose cumulant function can be written as
\begin{equation*}%\label{e:Clambda}
C\left\{ \zeta \ddagger \Lambda(A)\right\} =i\zeta m_{0}(A)-\frac{\zeta ^{2}}{2} m_{1}(A)  +\int_{\R}\left( e^{i\zeta x}-1-i\zeta \mathbf{1}_{[-1,1]}(x)\right) Q(A,dx),
\end{equation*}
where $m_{0}$ is a signed measure, $m_{1}$ is a positive measure and for every $A\in \mathcal{S}$, $Q(A,dx)$ is a measure on $\mathcal{B}(\R)$ without atom at $0$ such that $\int_{\R} \min \left\{ 1,x^{2}\right\} Q(A,dx)<\infty$. In this case we say that $\Lambda$ has the L\'{e}vy characteristics $(m_{0},m_{1},Q)$ and $Q$ is called the {\it generalized (deterministic) L\'{e}vy measure}. An important object in characterizing the class of non-random functions that are integrable with respect to $\Lambda$ is the \textit{control measure} $m$ defined as
\begin{equation*}
m(A)=\left| m_{0}\right| (A)+m_{1}(A)+\int_{\R}\min \left\{1,x^{2}\right\} Q(A,dx).
\end{equation*}%
%The integration of a function on $S$ with respect to the random measure $\Lambda$ can be defined first for real simple functions, then as a limit in probability of such integrals. 

The conditions for integrability of functions with respect to $\Lambda$ can be found in \cite{bn2001} and \cite{rajput1989spectral}. If function $f$ on $\R_+ \times \R$ is integrable with respect to the random measure $\Lambda$, then the cumulant function of the random variable $\int_{A}f d\Lambda$ is
\begin{equation}\label{integrationrule}
C\left\{ \zeta \ddagger \int_{A}f d\Lambda \right\} = \int_{A} \kappa_{L} (\zeta f(w))M(dw)
\end{equation}
where $\kappa_{L}$ is the cumulant function associated with the L\'evy basis $\Lambda$.
More details on  integration  can be found in \cite{rajput1989spectral}.

In defining the stationary supOU processes we will be interested in the homogeneous case where the characteristic triplet is of the form
$$
m_0=a M,\,  m_1=bM \quad \text{and} \quad  Q(dw,dx)=M(dw) \mu_L(dx),
$$
  where $a\in \R$, $b>0$, $\mu_L$ is a L\'evy measure and $M$ is a measure on $S$. Note that $M$ and $\mu_L$ are deterministic. Then   the cumulant function of the random variable $\Lambda (A)$ is
\begin{equation}\label{cumofHom}
C\left\{ \zeta \ddagger \Lambda(A)\right\} =M(A) \kappa_{L}(\zeta)
\end{equation}
where $\kappa_{L}$ is the cumulant function associated with the triplet $(a,b,\mu_L)$, i.e.
\begin{equation}\label{kappacumfun}
\kappa_{L}(\zeta) = i\zeta a -\frac{\zeta ^{2}}{2} b  +\int_{\R}\left( e^{i\zeta x}-1-i\zeta \mathbf{1}_{[-1,1]}(x)\right) \mu_L(dx).
\end{equation}
For more details see also \cite{fasen2007extremes,barndorff2011multivariate,barndorff2013multivariate,barndorff2013levy} where such measures are also referred to as L\'evy bases.
% For a generalization to Banach-space valued processes, see \cite{Mandrekar2008}.

\subsection{SupOU processes}

Although OU type processes provide a rich class of stationary models, their correlation structure is rather limited from the modeling perspective. On the other hand, superpositions of OU type processes introduced in \cite{bn2001} provide far more flexibility and can exhibit long-range dependence. They are obtained by randomizing the parameter $\lambda$ in
(\ref{xonRdef}), using a probability measure $\pi$ with support in $\R_+$.
The probability measure $\pi$  will affect the dependence structure.
   We  present basic facts about these processes following \cite{bn2001} and \cite{fasen2007extremes} (see also \cite{barndorff2013levy}).
  
 Suppose $\Lambda$ is a homogenous
infinitely divisible independently scattered random measures
 on $S=\R_+ \times \R$ such that \eqref{cumofHom} holds with $M=\pi \times Leb$ being the product of a probability measure $\pi$ on $\R_+$ and the Lebesgue measure on $\R$. We say that $(a,b,\mu_L,\pi)$ is the \textit{generating quadruple} (\cite{fasen2007extremes}) and the corresponding independently scattered random measure $\Lambda$ will be referred to as the \textit{L\'evy basis}.

The following result gives the existence of a superposition Ornstein-Uhlenbeck process; see \cite[Theorem 3.1]{bn2001}. We denote the points in $\R_+ \times \R$ as $w=(\xi, s)$ and $\Lambda(dw)=\Lambda(d\xi,ds)$.

\begin{theorem}\label{thm:existencesupOU}
Let $\kappa_X$ be the cumulant function of some selfdecomposable law, $(a,b,\mu_L)$ be the characteristic triplet of the associated BDLP with cumulant function $\kappa_L$ and let $\pi$ be a probability measure on $\R_+$. Define the L\'evy basis $\Lambda$ on $\R_+\times \R$ with generating quadruple $(a,b,\mu_L,\pi)$ and set
\begin{equation}\label{supOU}
X(t)=\int_{\R_+} e^{-\xi t}\int_{-\infty }^{\xi t}e^{s} \Lambda(d\xi,ds) = \int_{\R_+} \int_{\R} e^{-\xi t + s} \mathbf{1}_{[0,\infty)}(\xi t -s) \Lambda(d\xi,ds).
\end{equation}
Then $X=\{X(t), \, t\in \R\}$ is a well-defined, infinitely divisible and strictly stationary process. Moreover, for $t_1<\cdots < t_m$, the joint cumulant function
of $(X(t_1),\cdots,X(t_m)$ is
  $$
  C \left\{ \zeta_1, \dots, \zeta_m \ddagger \left(X(t_1),\dots , X(t_m) \right) \right\}
  $$
  \begin{equation}\label{cumfidissupOU}  
   = \int_{\R_+} \int_{\R} \kappa_L \left( \sum_{j=1}^m \mathbf{1}_{[0,\infty)} (\xi t_j-s) \zeta_j e^{-\xi t_j + s} \right) ds \, \pi(d\xi).
\end{equation}
In particular, since  $X=\{X(t), \, t\in \R\}$ is stationary,
\begin{equation*}
C\left\{ \zeta \ddagger X(t)\right\} = \kappa_X(\zeta),
\end{equation*}
and assuming that $X(t)$ has finite second moment, its correlation function is given by
\begin{equation}\label{corrsupOU}
r(\tau )=\int_{\R_+} e^{-\tau \xi }\pi (d\xi ), \quad \tau \geq 0.
\end{equation}
\end{theorem}

\begin{definition}
The process $X=\{X(t), \, t\in \R\}$ defined by \eqref{supOU} in Theorem \ref{thm:existencesupOU} is called a superposition Ornstein-Uhlenbeck (supOU) process.
\end{definition}

Relation \eqref{corrsupOU} is obtained by setting $m=2$ in \eqref{cumfidissupOU}, taking derivatives with respect to $\zeta_1$ and $\zeta_2$ and letting them tend to $0$. By comparing the definition of superposition \eqref{supOU} with the standard OU type process \eqref{xonRdef}, one can see the supOU process is obtained by randomizing the parameter $\lambda$ in \eqref{xonRdef} according to the probability measure $\pi$. A choice of $\pi$ will play an important role. Taking $\pi$ as in \eqref{regvarofpi} below will make $X$ long-range dependent.

\begin{remark} Here is a summary of the measures involved. The supOU process $X(t)$ in \eqref{supOU} is defined through an integral involving the random measure $\Lambda(d\xi, ds)$. For a fixed $t$, the corresponding cumulant function is
\begin{equation*}
    \kappa_X(\zeta) = C \left\{ \zeta \ddagger X(t) \right\} = \int_{\R_+} \int_{\R} \kappa_L \left( \mathbf{1}_{[0,\infty)} (\xi t-s) \zeta e^{-\xi t + s} \right) ds \, \pi(d\xi)
\end{equation*}
where $\kappa_L$ given in \eqref{kappacumfun} is associated with the L\'evy basis $\Lambda $ and involves the L\'evy measure $\mu_L$. The cumulant function $\kappa_X$ thus involves the corresponding deterministic measure
\begin{equation*}
  Q(dw,dx)=M(dw) \mu_L(dx) = \pi(d \xi) Leb(ds)  \mu_L(dx),
\end{equation*}
where $w=(\xi,s)$.
\end{remark}

\begin{remark}
In \cite{fasen2007extremes}, a supOU process is defined as
\begin{equation}\label{defsupOUlternative}
\widetilde{X}(t)=\int_{\R_+} \int_{\R} e^{-\xi (t-s)} \mathbf{1}_{[0,\infty)} (t-s) \widetilde{\Lambda}(d\xi,ds),
\end{equation}
where $\widetilde{\Lambda}$ has generating quadruple $(\widetilde{a},\widetilde{b},\widetilde{\mu}_L,\widetilde{\pi})$ such that $\rho:=\int_{\R_+}\xi^{-1} \widetilde{\pi}(d\xi) < \infty$. However, the two approaches are equivalent. Taking  $a=\rho \widetilde{a}$, $b=\rho \widetilde{b}$, $\mu_L=\rho \widetilde{\mu}_L$ and $\pi(d \xi)=\rho^{-1} \xi^{-1} \widetilde{\pi}(d\xi)$ in Theorem \ref{thm:existencesupOU}, we obtain a process which is a version of $\widetilde{X}$ defined in \eqref{defsupOUlternative} (see \cite[Proposition 2.1]{fasen2007extremes}).
\end{remark}

\begin{example}
If the measure $\pi$ in \eqref{cumfidissupOU} is degenerate such that $\pi\left(\{\lambda\}\right)=1$ for some $\lambda>0$, then it follows from \eqref{cumfidissupOU} that the finite dimensional distributions of $X$ are the same as for the standard OU type process \eqref{xonRdef}, that is
\begin{equation*}
  C \left\{ \zeta_1, \dots, \zeta_m \ddagger \left(X(t_1),\dots , X(t_m) \right) \right\} = \int_{\R} \kappa_L \left( \sum_{j=1}^m \mathbf{1}_{[0,\infty)} (\lambda t_j-s) \zeta_j e^{-\lambda t_j + s} \right) ds.
\end{equation*}
\end{example}

\begin{example}\label{exa:discretesup}
Suppose $\pi$ in \eqref{cumfidissupOU} is a discrete probability measure such that $\pi\left(\{\lambda_k\}\right)=p_k$, $k \in \N$ and $\lambda_k>0$. Then we have that
\begin{equation*}
  C \left\{ \zeta_1, \dots, \zeta_m \ddagger \left(X(t_1),\dots , X(t_m) \right) \right\} = \sum_{k=1}^\infty \int_{\R} p_k \kappa_L \left( \sum_{j=1}^m \mathbf{1}_{[0,\infty)} (\lambda_k t_j-s) \zeta_j e^{-\lambda_k t_j + s} \right) ds.
\end{equation*}
Thus in this case $X$ has the same distribution as the infinite discrete type superposition
\begin{equation*}
  \left\{\sum_{k=1}^\infty X^{(k)}(t), \ t \in \R \right\},
\end{equation*}
where $\{X^{(k)}(t), \, t \in \R\}$, $k \in \N$ are independent standard OU type processes corresponding to parameter $\lambda_k$ and BDLP with cumulant function $p_k \kappa_L$, $k \in \N$. In the case of finite second moment, such discrete type superposition is well defined in the sense of $L^2$ and a.s. convergence (see \cite{GLST2015}), and from \eqref{corrsupOU} the correlation function is
\begin{equation*}
  r(\tau)= \sum_{k=1}^\infty e^{-\lambda_k \tau} p_k, \quad \tau \geq 0.
\end{equation*}
%\textcolor{red}{In the notation of our paper it would be $\lambda_k=\lambda/k$ and $p_k = \zeta(1+2(1-H)) /k^{1+2(1-H)}$ with $\zeta$ the Riemann zeta function.}
\end{example}

By appropriate choices of probability measure $\pi$ one can achieve different correlation structures of the supOU processes. We will use the notation $f\sim g$ if $f(x)/g(x)\to 1$ as $x\to 0$ or $x\to \infty $. It follows from \eqref{corrsupOU} that the correlation function can be considered as the Laplace transform of $\pi$. Using Karamata's Tauberian theorem \cite[Theorem 1.7.1$'$]{bingham1989regular} one can easily obtain the following result (\cite{fasen2007extremes}).

\begin{prop}\label{prop:corrasymptotic}
Suppose $X$ is a square integrable supOU process with correlation function $r$, $L$ is a slowly varying function at infinity and $\alpha>0$. Then
\begin{equation}\label{regvarofpi}
  \pi \left( (0,x] \right) \sim L(x^{-1}) x^{\alpha}, \quad \text{ as } x \to 0
\end{equation}
if and only if
\begin{equation}\label{e:cdep}
  r( \tau) \sim \Gamma(1+\alpha) L(\tau) \tau^{-\alpha}, \quad \text{ as } \tau \to \infty.
\end{equation}
\end{prop}

The bigger the mass of $\pi$ is near origin, the slower is the decay of the correlation function at infinity. Hence, in view of (\ref{e:cdep}), if $\alpha \in (0,1)$ the correlation function is not integrable, and supOU process exhibits long-range dependence. We will denote
\begin{equation*}
\alpha=2 \overline{H}=2(1-H)
\end{equation*}
with $H$ as the long-range dependence parameter. Hence $\alpha \in (0,1)$ corresponds to $H \in (1/2, 1)$. More details on the dependence structure in specific examples can be found in \cite{barndorff2005spectral}.

\begin{example}\label{ex:pi:gamma}
Suppose $X$ is a supOU process such that $\pi $ is Gamma distribution with density
\begin{equation*}
  f(x)= \frac{1}{\Gamma(\alpha)} x^{\alpha-1} e^{-x} \mathbf{1}_{(0,\infty)}(x),
\end{equation*}
where $\alpha>0$. Then
\begin{equation*}
  \pi ((0,x])= \frac{\gamma(\alpha,x)}{\Gamma(\alpha)}, \quad x>0,
\end{equation*}
where $\gamma(\alpha,x)=\int_0^x u^{\alpha-1} e^{-u} du$ is the incomplete Gamma function. From the asymptotic expansion of $\gamma$ (\cite[Eq.~6.5.4 and Eq.~6.5.29]{abramowitz1964handbook}) we have that
\begin{equation*}
  \pi ((0,x]) \sim \frac{1}{\Gamma(\alpha+1)} x^{\alpha}, \quad \text{ as } x \to 0.
\end{equation*}
By Lemma \ref{prop:corrasymptotic} the correlation function has the property
\begin{equation*}
  r(\tau) \sim \tau^{-\alpha}, \quad \text{ as } \tau \to \infty.
\end{equation*}
In this case, we can explicitly compute from \eqref{corrsupOU} that
\begin{equation*}
  r(\tau) = \int_0^{\infty} e^{-\tau x} \frac{1}{\Gamma(\alpha)} x^{\alpha-1} e^{-x} dx =  (1+\tau)^{-\alpha} \frac{1}{\Gamma(\alpha)} \int_0^{\infty} x^{\alpha-1} e^{- x} dx = (1+\tau)^{-\alpha}.
\end{equation*}
Note that for $\alpha \in (0,1]$ the correlation function exhibits long-range dependence, while for $\alpha>1$ short-range dependence.
\end{example}

\begin{example}\label{ex:pi:ML}
If $\pi$ is the Mittag-Leffler distribution, then the correlation function of the supOU process is
\begin{equation*}
r(\tau) = (1+\tau^{\alpha})^{-1}, \quad 0<\alpha<2.
\end{equation*}
The supOU process obtained in this way is long-range dependent for $\alpha \in (0,1]$ and short-range dependent for $\alpha \in (1,2)$.
\end{example}

\begin{example}\label{ex:pi:ML2}
Another long-range dependent example can be obtained with $r(\tau)=E_{\alpha} (-\tau^\gamma)$, $\gamma \in (0,1)$, $\alpha \in (0,1)$ where
\begin{equation*}
  E_\alpha(z) = \sum_{k=0}^{\infty} \frac{z^k}{\Gamma(\alpha k+1)}, \quad z \in \mathbb{C},
\end{equation*}
is the Mittag-Leffler function. In this case
\begin{equation*}
r(\tau) \sim \frac{\tau^{-\gamma}}{\Gamma(1-\alpha)}, \quad \text{ as } \tau \to \infty.
\end{equation*}
See \cite[Example 4]{barndorff2005spectral} for details.
\end{example}

In our study of intermittency we will be concerned with the cumulant properties of integrated and partial sum process of supOU process. Tractable expressions for cumulant functions in both cases are established in the following subsections.

\subsection{Integrated process}
Suppose $X$ is a supOU process defined in \eqref{supOU} and let $X^*=\{X^*(t), \, t \geq 0\}$ be the integrated process
\begin{equation}\label{integratedsupOU}
  X^*(t) = \int_0^t X(s) ds.
\end{equation}
For $a,b\in \R$, let
\begin{equation*}
  \varepsilon(a,b)= \frac{1}{b} \left( 1-e^{-ab}\right)
\end{equation*}
and recall that $\kappa_{X^*}(\zeta, t)$ and $\kappa_{X^*}^{(m)}(t)$ denote the cumulant function and the $m$-th order cumulant of $X^*(t)$, respectively.

\begin{theorem}[Theorem 4.1 in \cite{bn2001}]
The cumulant function $\kappa_{X^*}$ of $X^*(t)$ satisfies
\begin{equation}\label{kappaX*t}
  \kappa_{X^*}(\zeta,t) = \zeta \int_0^\infty  \int_0^t \kappa_X'\left(\varepsilon(s, \xi) \zeta\right) ds \, \pi(d\xi),
  \end{equation}
  where $\kappa_X(\zeta)$ is the cumulant function of $X(1)$.
\end{theorem}

\begin{theorem}[Theorem 4.2 in \cite{bn2001}]\label{thm:integratedsupOU:cumulants}
Assume that $\kappa_X$ is analytic in a neighborhood of the origin. The cumulants of $X^*(t)$ are then given by
\begin{equation}\label{e:kIm}
  \kappa_{X^*}^{(m)}(t) = \kappa_X^{(m)} m I_{m-1}(t)
\end{equation}
where the $\kappa_X^{(m)}$ are the cumulants of $X(1)$,
\begin{equation}\label{Im-1}
  I_{m-1}(t) = \int_0^\infty \left( a_{m-1} + t \xi  + \sum_{k=1}^{m-1} (-1)^{k-1} {m-1 \choose k} \frac{1}{k} e^{-kt \xi} \right) \xi^{-m} \pi(d\xi)
\end{equation}
with
\begin{equation}\label{am-1}
  a_{m-1} = \sum_{k=1}^{m-1} (-1)^{k} {m-1 \choose k} \frac{1}{k}.
\end{equation}
\end{theorem}

The analyticity of the $\kappa_X$ in Theorem \ref{thm:integratedsupOU:cumulants} ensures the existence of all the cumulants of the marginal distribution of the underlying supOU process $X$. Note also that analyticity does not depend on the measure $\pi$ since the choice of $\pi$ does not affect the one-dimensional marginal distribution of $X$. The following is a useful criterion \cite[Theorem 7.2.1]{lukacs1970characteristic} for checking analyticity of the cumulant function.

\begin{lemma}\label{lemma:analiticity}
The characteristic and cumulant functions are analytic in a neighborhood of the origin if and only if there is a constant $C$ such that the corresponding distribution function $F$ satisfies
\begin{equation*}
  1-F(x)+F(-x) = O(e^{-u x}), \quad \text{ as } x\to \infty,
\end{equation*}
for all $0<u<C$.
\end{lemma}

It follows from Lemma \ref{lemma:analiticity} that the cumulant function of $X(t)$ is analytic in the neighborhood of the origin if there exists $a>0$ such that $\E e^{a |X(t)|} <\infty$. This implies in particular that all the moments and cumulants of $X(t)$ exist. This condition is satisfied for many  selfdecomposable distributions.

\begin{example}\label{ex:X:IG}
 The \textit{inverse Gaussian distribution} $IG(\delta,\gamma)$, $\gamma >0$, $\delta >0$ with density
\begin{equation*}
  f_{IG(\delta,\gamma)} (x) = \frac{\delta}{\sqrt{2\pi}} e^{\delta \gamma} x^{-3/2} \exp \left\{ - \frac{1}{2} \left( \delta^2 x^{-1} + \gamma^2 x \right) \right\} \mathbf{1}_{(0,\infty)}(x)
\end{equation*}
is selfdecomposable and hence, for any choice of probability measure $\pi$, there exists a supOU process $X$ with $IG(\delta,\gamma)$ stationary distribution. Since exponential moments are finite, the cumulant generating function is analytic in a neighborhood of the origin and has the form
\begin{equation*}
\kappa_X (\zeta)= \delta \left( \gamma - \sqrt{\gamma^2-2i\zeta}\right).
\end{equation*}
\end{example}

\begin{example}\label{ex:X:NIG}
The \textit{normal inverse Gaussian distribution} $NIG(\alpha,\beta,\delta,\mu)$ with parameters $\alpha \geq \left| \beta \right|$,  $\delta >0$, $\mu \in \mathbb{R}$ is another example of selfdecomposable distribution. The density of $NIG(\alpha,\beta,\delta,\mu)$ distribution satisfies (see \cite{barndorff1997processes})
\begin{equation*}
  f_{NIG(\alpha,\beta,\delta,\mu)} (x) \sim C | x|^{-3/2} e^{-\alpha | x|+\beta x}, \quad \text{ as } x\to \pm \infty.
\end{equation*}
Hence, there is $a>0$ such that $\E e^{a |X(t)|} <\infty$,  the cumulant generating function is analytic in a neighborhood of the origin and has the form
\begin{equation*}
\kappa_X (\zeta)=i\mu \zeta +\delta \left( \sqrt{\alpha^{2}-\beta ^{2}}-\sqrt{\alpha ^{2}-\left( \beta +i\zeta \right) ^{2}}\right).
\end{equation*}
\end{example}

Other examples of supOU processes satisfying conditions of Theorem \ref{thm:integratedsupOU:cumulants} can be obtained by taking the marginal distribution to be gamma, variance gamma, tempered stable, Euler’s
gamma, or z-distribution. See \cite{barndorff2005spectral} and \cite{GLST2015} for more details. On the other hand, the Student's $t$-distribution $T(\nu,\delta,\mu)$, $\nu>0$, $\delta>0$, $\mu\in \R$ whose density is 
\begin{equation*}
f_{T(\nu,\delta,\mu)}(x) =  \frac{\Gamma \left(\frac{\nu+1}{2}\right)}{\delta \Gamma\left(\frac{1}{2}\right) \Gamma\left(\frac{\nu}{2}\right)}
\left(1+\left( \frac{x-\mu}{\delta} \right)^2 \right)^{-\frac{\nu+1}{2}}, \quad x \in \mathbb{R},
\end{equation*}
provides an example of a self-decomposable distribution for which the cumulant function is not analytic around the origin since $\E |X|^q = \infty$ for $q>\nu$ (see e.g.~\cite{heyde2005student}).

It is worth noting that one can obtain expressions for cumulants without assuming  analyticity. In fact, taking derivatives with respect to $\zeta$ in \eqref{kappaX*t} and letting $\zeta\to 0$, one recovers the formula \eqref{e:kIm}. This approach can be used to investigate cumulants and moments when they exists only up to some finite order, as in the case of Student's distribution. In this paper we assume analyticity in order not to complicate the exposition.

\subsection{Partial sum process}
In addition to the integrated process, we also consider partial sums of a discretely sampled supOU process. Let
\begin{equation}\label{partialsumsupOU}
  X^+(t) = \sum_{i=1}^{\lfloor t \rfloor} X(i)
\end{equation}
and define
\begin{equation}\label{e:eta-ab}
  \eta(a,b)=e^{-b} \frac{1-e^{-ab}}{1-e^{-b}}.
\end{equation}

\begin{theorem}
The cumulant function $\kappa_{X^+}$ of $X^+(t)$ satisfies
\begin{equation}\label{kappa+}
  \kappa_{X^+}(\zeta,t) = \int_{0}^\infty \left( \sum_{k=1}^{\lfloor t \rfloor} \bigg( \kappa_X \left( e^{\xi} \eta\left(k, \xi\right)  \zeta \right) -\kappa_X \big( \eta\left(k, \xi\right)  \zeta \big) \bigg) + \kappa_X \big( \eta\left(\lfloor t \rfloor, \xi\right) \zeta \big) \right)\pi(d \xi),
\end{equation}
where $\kappa_X(\zeta)$ is the cumulant function of $X(1)$.
\end{theorem}

\begin{proof}
From \eqref{supOU}
\begin{align*}
X^+(t) &= \sum_{i=1}^{\lfloor t \rfloor} \int_{\R_+} \int_{\R} e^{-\xi i + s} \mathbf{1}_{(s/\xi,\infty)}(i) \Lambda(d\xi,ds)\\
&= \int_{\R_+} \int_{-\infty }^{0} \left(\sum_{i=1}^{\lfloor t \rfloor} e^{-\xi i +s}\right) \Lambda(d\xi,ds) + \int_{\R_+} \int_{0 }^{\xi \lfloor t \rfloor} \left(\sum_{i=\lfloor s/\xi \rfloor +1}^{\lfloor t \rfloor} e^{-\xi i +s}\right) \Lambda(d\xi,ds)\\
&= \int_{\R_+} \int_{-\infty }^{0} e^s \eta\left(\lfloor t \rfloor , \xi\right) \Lambda(d\xi,ds) + \int_{\R_+} \int_{0 }^{\xi \lfloor t \rfloor} \left(\sum_{i=\lfloor s/\xi \rfloor +1}^{\lfloor t \rfloor} e^{-\xi i +s}\right)  \Lambda(d\xi,ds).
\end{align*}
Using \eqref{integrationrule} and then \eqref{kappaXtoL} we get
\begin{align}
&\ \qquad \qquad \qquad \qquad \qquad \qquad
\kappa_{X^+}(\zeta, t) 
\nonumber \\ &= \int_{0}^\infty \int_{-\infty }^{0} \kappa_L \bigg( e^s \eta\big(\lfloor t \rfloor, \xi\big) \zeta \bigg) ds \, \pi(d \xi) + \int_{0}^\infty \int_{0 }^{\xi \lfloor t \rfloor} \kappa_L \left( \sum_{i=\lfloor s/\xi \rfloor +1}^{\lfloor t \rfloor} e^{-\xi i +s}  \zeta \right) ds \, \pi(d \xi)\nonumber \\
&= \int_{0}^\infty \int_{0}^{\infty} \kappa_L \bigg( e^{-s} \eta\big(\lfloor t \rfloor , \xi\big) \zeta \bigg) ds \, \pi(d \xi) + \int_{0}^\infty \int_{0 }^{\lfloor t \rfloor} \xi \kappa_L \left( \sum_{i=\lfloor u \rfloor +1}^{\lfloor t \rfloor} e^{-\xi i + \xi u}  \zeta \right) du \, \pi(d \xi) \nonumber \\
&= \int_{0}^\infty \kappa_X \bigg( \eta\big(\lfloor t \rfloor, \xi\big) \zeta \bigg) \pi(d \xi) + \int_{0}^\infty \int_{0 }^{\lfloor t \rfloor} \xi \kappa_L \left( \sum_{i=\lfloor u \rfloor +1}^{\lfloor t \rfloor} e^{-\xi i + \xi u}  \zeta \right) du \, \pi(d \xi).\label{theoremkappa+proof}
\end{align}
Note that all integrals in \eqref{theoremkappa+proof} are finite because the cumulant function $\kappa_L$ is absolutely integrable with respect to the control measure, see \cite[Proposition 2.6]{rajput1989spectral}. For the second integral on the right, by computing the partial sum of the geometric sequence
\begin{equation*}
  \sum_{i=k +1}^{\lfloor t \rfloor} e^{-\xi i} = e^{-\xi (k+1)} \frac{1-e^{-\xi(\lfloor t \rfloor - k)}}{1-e^{-\xi}} = e^{-\xi k} \eta\big(\lfloor t \rfloor -k, \xi\big)
\end{equation*}
we have
%\begin{align*}
$$
\int_{0}^\infty \int_{0 }^{\lfloor t \rfloor}
 \xi \kappa_L \left( \sum_{i=\lfloor u \rfloor +1}^{\lfloor t \rfloor} e^{-\xi i + \xi u}  \zeta \right) du \, \pi(d \xi)
 $$
 $$
  = \int_{0}^\infty \xi \sum_{k=0}^{\lfloor t \rfloor -1}  \int_{k}^{k+1} \kappa_L \left( e^{\xi u} \sum_{i=k +1}^{\lfloor t \rfloor} e^{-\xi i}  \zeta \right) du \, \pi(d \xi)
 $$
 $$
=\int_{0}^\infty \xi \sum_{k=0}^{\lfloor t \rfloor -1}  \int_{k}^{k+1} \kappa_L \bigg( e^{-\xi (k-u)} \eta\big(\lfloor t \rfloor -k, \xi\big)  \zeta \bigg) du \, \pi(d \xi).
$$
The change of variables $s=k-u+1$ and \eqref{kappaLtoX} yield
\begin{align*}
\int_{0}^\infty \int_{0 }^{\lfloor t \rfloor} & \xi \kappa_L \left( \sum_{i=\lfloor u \rfloor +1}^{\lfloor t \rfloor} e^{-\xi i + \xi u}  \zeta \right) du \, \pi(d \xi)\\
& = \int_{0}^\infty \xi \sum_{k=0}^{\lfloor t \rfloor -1}  \int_{0}^{1} \kappa_L \bigg( e^{-\xi (s-1)} \eta\big(\lfloor t \rfloor -k, \xi\big)  \zeta \bigg) ds \, \pi(d \xi)\\
&= \int_{0}^\infty \xi \sum_{k=0}^{\lfloor t \rfloor -1}  \int_{0}^{1} e^{-\xi (s-1)} \eta\big(\lfloor t \rfloor -k, \xi\big)  \zeta  \kappa_X' \bigg( e^{-\xi (s-1)} \eta\big(\lfloor t \rfloor -k, \xi\big)  \zeta \bigg) ds \, \pi(d \xi)\\
&= \int_{0}^\infty \xi \sum_{k=0}^{\lfloor t \rfloor -1}  \int_{0}^{1} \frac{d}{ds} \left[ - \frac{1}{\xi} \kappa_X \bigg( e^{-\xi (s-1)} \eta\big(\lfloor t \rfloor -k, \xi\big)  \zeta \bigg) \right] ds \, \pi(d \xi)\\
&= \int_{0}^\infty \sum_{k=0}^{\lfloor t \rfloor -1} \left( \kappa_X \bigg( e^{\xi} \eta\big(\lfloor t \rfloor -k, \xi\big)  \zeta \bigg) -\kappa_X \bigg( \eta\big(\lfloor t \rfloor -k, \xi\big)  \zeta \bigg) \right) \pi(d \xi)\\
&= \int_{0}^\infty \sum_{k=1}^{\lfloor t \rfloor} \left( \kappa_X \bigg( e^{\xi} \eta\big(k, \xi\big)  \zeta \bigg) -\kappa_X \bigg( \eta\big(k, \xi\big)  \zeta \bigg) \right) \pi(d \xi).
\end{align*}
Combining this with \eqref{theoremkappa+proof} yields \eqref{kappa+}.
\end{proof}

\begin{theorem}\label{thm:partialsumsupOU:cumulants}
Assume that the cumulant function $\kappa_X$ of $X(t)$ is analytic in a neighborhood of the origin. The cumulants of $X^+(t)$ are then given by
\begin{equation*}
  \kappa_{X^+}^{(m)}(t) = \kappa_X^{(m)} J_{m-1}(t)
\end{equation*}
where the $\kappa_X^{(m)}$ are the cumulants of $X(1)$ and
\begin{equation}\label{Jm-1}
\begin{aligned}
  J_{m-1}(t) &= \int_{0}^\infty \Bigg( \left(1 - e^{-m\xi}\right) \left(\lfloor t \rfloor - 1\right) + \left(1 - e^{-m\xi}\right) \sum_{j=1}^m {m \choose j} (-1)^j e^{-j\xi} \frac{1-e^{-j \left(\lfloor t \rfloor -1 \right)\xi}}{1-e^{-j \xi}} &\\
  & \qquad  + \left(1-e^{-\lfloor t \rfloor \xi}\right)^m \Bigg) \frac{1}{\left(1-e^{-\xi}\right)^m} \pi(d \xi).
\end{aligned}
\end{equation}
\end{theorem}

\begin{proof}
Using \eqref{e:cgf} and \eqref{kappa+}, expand the cumulant function of $X$ to get
\begin{align*}
 &\ \qquad \qquad \qquad \qquad \qquad \qquad
   \kappa_{X^+}(\zeta,t) \\
    &= \int_{0}^\infty \left( \sum_{k=1}^{\lfloor t \rfloor} \bigg( \sum_{m=1}^\infty \kappa_X^{(m)} \frac{\left(i \zeta \eta\big(k,\xi \big) \right)^{m}}{m!} \left(e^{m\xi} -1 \right)\bigg) + \sum_{m=1}^\infty \kappa_X^{(m)} \frac{\left(i \zeta \eta\big(\lfloor t \rfloor,\xi \big) \right)^{m}}{m!} \right)\pi(d \xi)\\
  &= \sum_{m=1}^\infty \kappa_X^{(m)} \frac{\left(i \zeta \right)^{m}}{m!} \int_{0}^\infty \left( \sum_{k=1}^{\lfloor t \rfloor} \eta\big(k,\xi \big)^m \left(e^{m\xi} -1 \right) + \eta\big(\lfloor t \rfloor,\xi \big)^m \right) \pi(d \xi)
\end{align*}
and by identifying the coefficients in the expansion, we get $\kappa_{X^+}^{(m)}(t)=\kappa_X^{(m)} J_{m-1}(t)$, where
\begin{equation*}
  J_{m-1}(t) = \int_{0}^\infty \left( \sum_{k=1}^{\lfloor t \rfloor} \eta\big(k,\xi \big)^m \left(e^{m\xi} -1 \right) + \eta\big(\lfloor t \rfloor,\xi \big)^m \right) \pi(d \xi).
\end{equation*}
Use \eqref{e:eta-ab} to get
\begin{align}
 & \qquad \qquad \qquad \qquad \qquad \qquad J_{m-1}(t) \nonumber\\
   &= \int_{0}^\infty \left( \sum_{k=1}^{\lfloor t \rfloor} e^{-m \xi} \frac{\left(1-e^{-k\xi}\right)^m}{\left(1-e^{-\xi}\right)^m}  \left(e^{m\xi} -1 \right) + e^{-m \xi} \frac{\left(1-e^{-\lfloor t \rfloor \xi}\right)^m}{\left(1-e^{-\xi}\right)^m}  \right) \pi(d \xi) \nonumber \\
  &= \int_{0}^\infty \frac{1}{\left(1-e^{-\xi}\right)^m} \left( \left(1 - e^{-m\xi}\right) \sum_{k=1}^{\lfloor t \rfloor - 1} \left(1-e^{-k\xi}\right)^m + \left(1-e^{-\lfloor t \rfloor \xi}\right)^m \right) \pi(d \xi) \label{Jm-1-proof}\\
  &= \int_{0}^\infty \left( \left(1 - e^{-m\xi}\right) \sum_{k=1}^{\lfloor t \rfloor - 1} \sum_{j=0}^m {m \choose j} (-1)^j e^{-jk\xi} + \left(1-e^{-\lfloor t \rfloor \xi}\right)^m \right) \frac{1}{\left(1-e^{-\xi}\right)^m} \pi(d \xi) \nonumber \\
  &= \int_{0}^\infty \Bigg( \left(1 - e^{-m\xi}\right) \left(\lfloor t \rfloor - 1\right) + \left(1 - e^{-m\xi}\right) \sum_{j=1}^m {m \choose j} (-1)^j e^{-j\xi} \frac{1-e^{-j \left(\lfloor t \rfloor -1 \right)\xi}}{1-e^{-j \xi}} \nonumber \\
  & \qquad  + \left(1-e^{-\lfloor t \rfloor \xi}\right)^m \Bigg) \frac{1}{\left(1-e^{-\xi}\right)^m} \pi(d \xi). \nonumber
\end{align}
\end{proof}

\section{Intermittency of integrated and partial sum process}\label{sec4}

In this section we establish asymptotic properties of cumulants and moments of the integrated supOU process $X^*$ defined in \eqref{integratedsupOU} and the partial sum process $X^+$ defined in \eqref{partialsumsupOU}. The underlying supOU process will be assumed to have a power law decay of the correlation function, which can be achieved with the appropriate choice of the probability measure $\pi$, as given by Proposition \ref{prop:corrasymptotic}. In the case of long-range dependence, we will show that both variants of cumulative processes can be intermittent. Before doing that, we provide examples where asymptotic normality easily follows.

\begin{example}\label{ex:finitesuplimit}
Consider a supOU process from Example \ref{exa:discretesup} such that $\pi$ is a discrete probability measure with finite support $\{\lambda_k : k=1,\dots K\}$ and $\pi(\{\lambda_k\})=p_k$. In this case, supOU process has the same distribution as the finite superposition $X=\{X(t), \, t \in \R\}$ defined by
\begin{equation*}
  X(t)=\sum_{k=1}^K X^{(k)}(t),
\end{equation*}
where $\{X^{(k)}(t), \, t \in \R\}$, $k = 1,\dots,K$ are independent standard OU type processes corresponding to parameter $\lambda_k$ and BDLP with cumulant function $p_k \kappa_L$, $k = 1,\dots,K$. Suppose $\E |X(1)|^{2+\delta} < \infty$ for some $\delta>0$ and let $\{S(t), \, t \geq 0\}$ denote the centered partial sum process
\begin{equation*}
  S(t) = \sum_{i=1}^{\lfloor t \rfloor} \left( X(i) - \E X(i) \right).
\end{equation*}
Each OU type process $\{X^{(k)}(t), \, t \in \R\}$, $k = 1,\dots,K$ satisfies the strong mixing property with an exponentially decaying rate of mixing coefficients (\cite{masuda2004}), and so does a sequence $X(i)$, $i \in \N$ as a finite sum of these processes. Application of the invariance principle for strong mixing sequences (\cite{davydov1968convergence}; see also \cite{oodaira1972functional}) shows that
\begin{equation*}
  \frac{S(nt)}{\sigma \sqrt{n}} \Rightarrow B(t), \quad t \in [0,1],
\end{equation*}
as $n \to \infty$, where $\{B(t), \, t \in [0,1]\}$ is a Brownian motion, $\sigma$ positive constant and the convergence is weak convergence in Skorokhod space $D[0,1]$. In particular, \eqref{thm:limitdistr} holds with $Y$ being the partial sum process and for every $t\in [0,1]$
\begin{equation*}
  \frac{S(nt)}{\sigma \sqrt{n}} \overset{d}{\to} \mathcal{N}(0,t),
\end{equation*}
as $n \to \infty$. If $q>2$ is such that $\E |X(1)|^{q} < \infty$, then by the result of \cite{yokoyama1980moment}, $q$-th absolute moment of $S(nt)/(\sigma \sqrt{n})$ converges to that of $\mathcal{N}(0,t)$. Then by Theorem \ref{theorem:limit} the scaling function of the partial sum process $S(t)$ is $\tau_S(q)=q/2$, and there is no intermittency.
\end{example}

\begin{example}\label{ex:gaussian}
Let $\{X(t), \, t\ge 0\}$ be a \textit{Gaussian supOU process}, that is a supOU process with the generating quadruple $(0,\sigma^2,0,\pi)$ where $\sigma^2>0$ and $\pi$ is a probability measure. One can check from \eqref{cumfidissupOU} that $X$ is indeed a Gaussian process with zero mean. Suppose further that $\pi$ satisfies \eqref{regvarofpi} for some $\alpha>0$ so that the correlation function satisfies \eqref{e:cdep}. Let  $X^+(t) = \sum_{i=1}^{\lfloor t \rfloor} X(i)$ be the corresponding partial sum process.

When $\alpha <1$, long-range dependence is present, and from \cite[Lemma 5.1]{Taqqu1975}, the normalized partial sum process
\begin{equation*}
\frac {1}{n^H \sqrt{L(n)}}X^+(nt)
\end{equation*}
with $H=1-\alpha /2$, converges in Skorokhod space $D[0,1]$ to a process that is fractional Brownian motion with Hurst parameter $H$ up to a multiplicative constant. The partial sum $X^+(t)$ is a mean zero Gaussian random variable with the variance satisfying $\E \left( X^+(t)\right)^2 \sim C \lfloor t \rfloor^{2H} L \left( \lfloor t \rfloor \right)$ (see the proof of \cite[Lemma 5.1]{Taqqu1975}). Since the $q$-th absolute moment of a Gaussian distribution is proportional to the $q$-th power of the standard deviation, it follows that $\tau_{X^+}(q)=Hq$, and there is no intermittency.

If $\alpha >1$, then the variance of $X^+(t)$ is of the order $t^{1/2}$, and the finite-dimensional distributions of
\begin{equation*}
\frac {1}{n^{1/2}} X^+(nt)
\end{equation*}
converge to those of the Brownian motion, see \cite[Theorem 2.3.1]{LeonenkoIvanov}. In the case $\alpha=1$, the limit is also Gaussian with an extra factor of a slowly varying function in the variance and in the normalizing sequence of the partial sum, see \cite[Theorem 2.3.2]{LeonenkoIvanov}. The same argument as in the case $\alpha <1$ shows that the scaling function is $\tau_{X^+}(q)=q/2$, and there is no intermittency.
\end{example}

To show that integrated supOU process $X^*(t)=\int _0^t X(s)ds$ can be intermittent, we first establish the form of the cumulant based scaling function $\sigma_{X^*}(m)$ defined in \eqref{defsigma}. Recall that $\kappa_X^{(m)}$ denotes the $m$-th cumulant of $X(t)$. In particular, $\kappa_X^{(1)}=\E X(t)$.

\begin{lemma}\label{thm:intermittency-integrated}
Suppose that the stationary supOU process $X$ defined in \eqref{supOU} satisfies the conditions of Proposition \ref{prop:corrasymptotic} and satisfies \eqref{regvarofpi} with some $\alpha>0$. Further, suppose that  $\kappa_X$ is analytic in a neighborhood of the origin and let $\sigma_{X^*}$ be the cumulant based scaling function \eqref{defsigma} of the integrated process $\{X^*(t), \, t \geq 0\}$. If the mean $\kappa_X^{(1)}\neq 0$, then
\begin{equation*}
\sigma_{X^*}(1)=1.
\end{equation*}
For every $m>\alpha+1$ such that $\kappa_X^{(m)}\neq 0$, we have
\begin{equation*}
  \sigma_{X^*}(m) = m - \alpha.
\end{equation*}
\end{lemma}

\begin{proof}
By Theorem \ref{thm:integratedsupOU:cumulants} we have that
\begin{equation}\label{proof:thm:intermittency-integrated:1}
  \sigma_{X^*}(m) = \lim_{t\to \infty} \frac{\log \left| \kappa_{X^*}^{(m)}(t) \right|}{\log t} = \lim_{t\to \infty} \frac{\log \left| \kappa_X^{(m)} m I_{m-1}(t) \right|}{\log t} = \lim_{t\to \infty} \frac{\log \left| I_{m-1}(t) \right|}{\log t}.
\end{equation}
From the expression \eqref{Im-1} for $I_{m-1}(t)$ we obtain the following form
\begin{align}
  I_{m-1}(t) &= \int_0^\infty \left( a_{m-1} + t \xi  + \sum_{k=1}^{m-1} (-1)^{k-1} {m-1 \choose k} \frac{1}{k} e^{-kt \xi} \right) \xi^{-m} \pi(d\xi) \label{proof:thm:intermittency-integrated:1a}\\
  &= \int_0^\infty \int_0^{\xi t} \left( 1 + \sum_{k=1}^{m-1} (-1)^{k} {m-1 \choose k}  e^{-kw} \right) dw \xi^{-m} \pi(d\xi), \nonumber
\end{align}
since
\begin{equation*}
	\int _0^{\xi t}e^{-k\omega }d\omega =-\frac{1}{k}\left(e^{-kt\xi }-1\right).
\end{equation*}
Hence
\begin{align*}
  I_{m-1}(t) &= \int_0^\infty \int_0^{\xi t} \left(1- e^{-w} \right)^{m-1} dw \xi^{-m} \pi(d\xi)\\
  &= \int_0^\infty \left(1- e^{-w} \right)^{m-1} \int_{w/t}^{\infty} \xi^{-m} \pi(d\xi) dw.
\end{align*}
First, if $m=1$ then \eqref{proof:thm:intermittency-integrated:1a} implies $I_0(t)=\int_0^\infty t \pi(d\xi)=t$ since $\pi$ is a probability measure. Hence \eqref{proof:thm:intermittency-integrated:1} yields $\sigma_{X^*}(1)=1$.

Now suppose $m\geq 2$. Since $\pi \left( (0,x] \right) \sim L(x^{-1}) x^{\alpha}$ as $x \to 0$, by putting $\widetilde{\pi}=\pi \circ g$ with $g(\xi)=1/\xi$ we obtain a probability measure that is regularly varying at infinity, more precisely
\begin{equation*}
  \widetilde{\pi} \left( (u,\infty) \right) \sim L(u) u^{-\alpha}, \quad \text{ as } u \to \infty.
\end{equation*}
A variant of Karamata's theorem for Lebesgue-Stieltjes integrals \cite[Theorem VIII.9.2.]{feller1971theory} gives
\begin{equation}\label{proof:thm:intermittency-integrated:21}
  \int_{0}^{t} u^{m} \widetilde{\pi}(du) \sim \frac{\alpha}{m-\alpha} L(t) t^{m-\alpha}, \quad \text{ as } t \to \infty.
\end{equation}
This result can be understood heuristically by supposing that $\widetilde{\pi}$ has a density $\widetilde{\pi} \left( du \right) \sim L(u) \alpha u^{-\alpha-1}du $ as $u \to \infty$. Since the integral $\int_{0}^{t} u^{m} \widetilde{\pi}(du)$ is regularly varying function at infinity in $t$, it can be written in the form
\begin{equation}\label{proof:thm:intermittency-integrated:2}
  \int_{0}^{t} u^{m} \widetilde{\pi}(du) = \frac{\alpha}{m-\alpha} L_1(t) t^{m-\alpha},
\end{equation}
with $L_1$ slowly varying at infinity such that $L_1(t) \sim L(t)$ as $t \to \infty$. Now by the change of variables $u=1/\xi$
\begin{equation*}
  \int_{w/t}^{\infty} \xi^{-m} \pi(d\xi) = \int_{0}^{t/w} u^{m} \widetilde{\pi}(du) = \frac{\alpha}{m-\alpha} L_1(t/w) (t/w)^{m-\alpha},
\end{equation*}
and so
\begin{align}
  I_{m-1}(t) &= \frac{\alpha}{m-\alpha} t^{m-\alpha} \int_0^\infty L_1(t/w)\left(1- e^{-w} \right)^{m-1} w^{\alpha-m} dw \nonumber\\
  &= \frac{\alpha}{m-\alpha} t^{m-\alpha} \int_0^\infty L_1(t z) \left(1- e^{-\frac{1}{z}} \right)^{m-1} z^{m-\alpha-2} dz \label{e:Im3}.
\end{align}
To show that the integral on the right varies slowly in $t$, we split it into two parts and use \cite[Proposition 4.1.2]{bingham1989regular}. The function $(1-e^{-1/z})\sim z^{-1}$ as $z\to \infty$ and hence
\begin{equation*}
  f(z)=\left(1- e^{-\frac{1}{z}} \right)^{m-1} z^{m-\alpha-2}
\end{equation*}
is regularly varying at infinity with index $-\alpha-1$ and regularly varying at zero with index $m-\alpha-2$. Due to the assumption $m>\alpha+1$, we can choose $0<\delta<m-\alpha-1$ and
\begin{equation}\label{proof:thm:intermittency-integrated:delta0}
  \int_0^1 z^{-\delta} f(z) dz < \infty.
\end{equation}
From \eqref{proof:thm:intermittency-integrated:2} we have that
\begin{equation*}
  L_1(t) = \frac{m-\alpha}{\alpha} t^{\alpha-m} \int_0^t u^m \widetilde{\pi}(du) \leq \frac{m-\alpha}{\alpha} t^{\alpha},
\end{equation*}
since $\widetilde{\pi}$ is a probability measure. Hence $t^{\delta}L_1(t)$ is locally bounded on $[0,\infty)$. By applying \cite[Proposition 4.1.2(a)]{bingham1989regular} it follows that
\begin{equation*}
  \int_0^1 L_1(t z) f(z) dz  \sim L_1(t) \int_0^1 f(z) dz, \quad \text{ as } t \to \infty.
\end{equation*}
On the other hand, for $0<\delta<\alpha$
\begin{equation*}
  \int_1^\infty z^{\delta} f(z) dz < \infty
\end{equation*}
and by application of \cite[Proposition 4.1.2(b)]{bingham1989regular} we obtain
\begin{equation*}
  \int_1^\infty L_1(t z) f(z) dz  \sim L_1(t) \int_1^\infty f(z) dz, \quad \text{ as } t \to \infty.
\end{equation*}
Going back to \eqref{e:Im3}, we have
\begin{equation*}
  I_{m-1}(t) \sim \frac{\alpha}{m-\alpha} t^{m-\alpha} L_1(t) \int_0^\infty \left(1- e^{-\frac{1}{z}} \right)^{m-1} z^{m-\alpha-2} dz
\end{equation*}
and from \eqref{proof:thm:intermittency-integrated:1} we get
\begin{equation}\label{proof:thm:intermittency-integrated:3}
  \sigma_{X^*}(m) =\lim_{t\to \infty} \frac{\log \left| I_{m-1}(t) \right|}{\log t}= m-\alpha
\end{equation}
since due to slow variation of $L_1$, $\log L_1(t) / \log t \to 0$ as $t\to \infty$.
\end{proof}

Using the relation between cumulants and moments we can now obtain the corresponding asymptotic behavior of the moments. This will yield intermittency as defined in \eqref{intermittency}. In central limit type theorems with finite variance one supposes that the mean is zero. We shall do this here as well and thus set the first cumulant $\kappa_X^{(1)}= 0$.

\begin{theorem}\label{cor:intermittency-integrated}
Suppose that for the non-Gaussian supOU process $X$ the assumptions of Lemma \ref{thm:intermittency-integrated} hold with $\alpha>0$, $\kappa_X^{(1)}= 0$ and $\kappa_X^{(2)} \neq 0$. If $\tau_{X^*}$ is the scaling function \eqref{deftau} of $X^*=\{X^*(t), \, t \geq 0\}$, then for every $q\geq q^*$
\begin{equation*}
	\tau_{X^*}(q) = q-\alpha,
	\end{equation*}
where $q^*$ is the smallest even integer greater than $2\alpha$. In particular, for $q^*\leq p < r$
	\begin{equation*}
	\frac{\tau_{X^*}(p)}{p} < \frac{\tau_{X^*}(r)}{r}
	\end{equation*}
	and hence $X^*$ is intermittent.
\end{theorem}

\begin{proof}
The marginal distribution of $X$ is selfdecomposable and hence infinitely divisible. Since it is not Gaussian, the L\'evy measure is non-null and by \cite[Remark 3.4.]{gupta2009cumulants} we have that for every even $m$, $\kappa_X^{(m)} \neq 0$. Using the expression for moment in terms of cumulants (see e.g.~\cite[Proposition 3.3.1]{peccati2011wiener}), for an even integer $m$ we have
\begin{equation}\label{proof:cor:intermittency-integrated}
  E|X^*(t)|^m = E(X^*(t))^m = \sum_{k=1}^m B_{m,k} \left( \kappa_{X^*}^{(1)}(t), \dots, \kappa_{X^*}^{(m-k+1)}(t) \right),
\end{equation}
where $B_{m,k}$ is the partial Bell polynomial given by (see \cite[Definition 2.4.1]{peccati2011wiener})
$$ B_{m,k}(x_1,\dots,x_{m-k+1})
$$
\begin{equation}\label{proof:cor:intermittency-integrated1}
   = \sum_{r_1,\dots,r_{m-k+1}} \frac{m!}{r_1! \cdots r_{m-k+1}!} \left(\frac{x_1}{1!}\right)^{r_1} \cdots \left(\frac{x_{m-k+1}}{(m-k+1)!}\right)^{r_{m-k+1}}
\end{equation}
and the sum is over all nonnegative integers $r_1,\dots,r_{m-k+1}$ satisfying 
\begin{equation}\label{proof:cor:intermittency-integrated2}
r_1+\cdots +r_{m-k+1}=k
\end{equation}
and 
\begin{equation}\label{proof:cor:intermittency-integrated3}
1 r_1 + 2 r_2 + \cdots + (m-k+1) r_{m-k+1}=m. 
\end{equation}

For $l> \alpha+1$ such that $\kappa_{X}^{(l)}\neq 0$, we have from the proof of Lemma \ref{thm:intermittency-integrated} that $\kappa_{X^*}^{(l)}(t) \sim L_l (t) t^{l-\alpha}$ as $t\to \infty$ with $L_l$ slowly varying at infinity. On the other hand, if $\kappa_{X}^{(l)}= 0$, then also $\kappa_{X^*}^{(l)}(t)=0$ by \eqref{e:kIm}. Since by the assumption $\kappa_{X^*}^{(1)}(t)=0$, the nonzero terms of the sum in the expression for $B_{m,k} \left(  \kappa_{X^*}^{(1)}(t), \dots, \kappa_{X^*}^{(m-k+1)}(t) \right)$ are obtained when $r_1=0$. 

\textbf{Case $\alpha<1$}. Assume for the moment that $\alpha <1$ so that the previous discussion applies for any $l \geq 2$. Now we can write
$$
B_{m,k} \left( \kappa_{X^*}^{(1)}(t), \dots, \kappa_{X^*}^{(m-k+1)}(t) \right)
$$
\begin{align} 
  &\sim \sum_{r_2,\dots,r_{m-k+1}} L_{r_2,\dots,r_{m-k+1}}(t) t^{(2-\alpha) r_2} \cdots t^{(m-k+1-\alpha) r_{m-k+1}}\nonumber \\
&= \sum_{r_2,\dots,r_{m-k+1}} L_{r_2,\dots,r_{m-k+1}}(t) t^{2 r_2 + \cdots + (m-k+1) r_{m-k+1} - \alpha \left( r_2 + \cdots + r_{m-k+1} \right)}\nonumber \\
&= \sum_{r_2,\dots,r_{m-k+1}} L_{r_2,\dots,r_{m-k+1}}(t) t^{m - \alpha k},\label{e:proof:cor:intermittency-integrated:11}
\end{align}
where $L_{r_2,\dots_{m-k+1}}(t)$ are slowly varying functions coming from the product of powers of $L_1,\dots,L_{m-k+1}$. If one of the cumulants $\kappa_{X^*}^{(3)}(t), \dots, \kappa_{X^*}^{(m-k+1)}(t)$ is zero, say $\kappa_{X^*}^{(l)}(t)$, then \eqref{e:proof:cor:intermittency-integrated:11} should be understood in the sense that the term in the sum is zero unless $r_l=0$. Since $\kappa_X^{(m)} \neq 0$, the sum \eqref{e:proof:cor:intermittency-integrated:11} for $k=1$ contains at least one term of the form $L_{r_2,\dots,r_{m}}(t) t^{m - \alpha}$. Finally then from \eqref{proof:cor:intermittency-integrated} we have that for some slowly varying function $\widetilde{L}$
\begin{equation}\label{proof:conclusion1}
E|X^*(t)|^m \sim \widetilde{L}(t) t^{m-\alpha}
\end{equation}
and so $\tau_{X^*}(m) = m-\alpha$ for any even integer $m\geq 2$.

\medskip\textbf{Case $\alpha>1$, $\alpha\notin\N$}. Now suppose $\alpha\geq 1$ and $m$ is an even integer greater than $2 \alpha$. Again, the term for $k=1$ in the sum \eqref{proof:cor:intermittency-integrated} would contain $L_{r_2,\dots,r_{m}}(t) t^{m - \alpha}$. It remains to show that the terms involving cumulants of order $j \in \{2,\dots,\lfloor \alpha +1 \rfloor \}$ will not dominate the $t^{m - \alpha}$ term. Indeed, for  $j \in \{2,\dots,\lfloor \alpha +1 \rfloor \}$ we have that $\int_0^\infty \xi^{-j+1} \pi(d \xi) < \infty$ and from \eqref{e:kIm} and \eqref{proof:thm:intermittency-integrated:1a} it follows that
\begin{eqnarray}\label{e:boundtcumulants}
\left| \kappa_{X^*}^{(j)}(t) \right| 
&=& \left| \kappa_X^{(j)} \right| j \left| \int_0^\infty \int_0^{t} \left(1- e^{-\xi s} \right)^{j-1} ds \xi^{-j+1} \pi(d\xi) \right|
\nonumber \\
&\leq& t  \left| \kappa_X^{(j)} \right| j  \int_0^\infty \xi^{-j+1} \pi(d\xi) =: C_j t .
\end{eqnarray}
Considering the terms appearing in the sum \eqref{proof:cor:intermittency-integrated1} and using Lemma \ref{thm:integratedsupOU:cumulants} and \eqref{e:boundtcumulants}, one can see that, up to slowly varying function, each term can be bounded by the following power of $t$
\begin{equation*}
t^{r_2} \cdots t^{r_{\lfloor \alpha +1 \rfloor}} \left(t^{\lfloor \alpha +1 \rfloor +1 -\alpha} \right)^{r_{\lfloor \alpha +1 \rfloor +1}} \cdots \left(t^{m-k +1 -\alpha} \right)^{r_{m-k +1}}
\end{equation*}
with nonnegative integers $r_1,\dots,r_{m-k+1}$ satisfying \eqref{proof:cor:intermittency-integrated2} and \eqref{proof:cor:intermittency-integrated3}. One will get the highest power of $t$ by setting all the $r_j's$ to zero but one, so that $r_j\leq m/j$ for $j \in \{2,\dots,\lfloor \alpha +1 \rfloor \}$. Since $j \in \{2,\dots,\lfloor \alpha +1 \rfloor \}$, the highest value is achieved when $j=2$, corresponding to the exponent $m/2$. Hence, the dominant term as $t\to \infty$ coming from cumulants of order $j \in \{2,\dots,\lfloor \alpha +1 \rfloor \}$ would be $\left(\kappa_{X^*}^{(j)}(t)\right)^{m/2} \leq C t^{m/2}$. If $m/2<m-\alpha$, then the term containing $L_{r_2,\dots,r_{m}}(t) t^{m - \alpha}$ would dominate the term bounded by $t^{m/2}$. But this holds, since $m-\alpha>m/2 \Leftrightarrow m>2\alpha$ and hence we proved \eqref{proof:conclusion1} for any even integer greater than $2 \alpha$.

\medskip
\textbf{Case $\alpha=1,2,\cdots$}. The problem may appear with $j=\lfloor \alpha +1 \rfloor=\alpha+1$ but the argument will go along the same lines as the case $\alpha>1$, $\alpha \notin \N$. If $\int_0^\infty \xi^{-1} \pi(d \xi) < \infty$, then the argument applies unchanged. Suppose that $\int_0^\infty \xi^{-1} \pi(d \xi) = \infty$. For $\kappa_{X^*}^{(\alpha+1)}(t)$,  \eqref{proof:thm:intermittency-integrated:21} holds. Given $\varepsilon>0$ we can take $t$ large enough so that 
\begin{equation*}
|\kappa_{X^*}^{(\alpha+1)}(t)| \leq C t^{2-\alpha+\varepsilon} \leq C t^{1+\varepsilon}.
\end{equation*}
By the same argument as in the proof of case $\alpha>1$, we would have that if we take $\varepsilon$ small enough, then the term containing $L_{r_2,\dots,r_{m}}(t) t^{m - \alpha}$ dominates the term bounded by $t^{(1+\varepsilon)m/2}$. Hence, \eqref{proof:conclusion1} holds when $\alpha\in \N$ for every even integer $m>2 \alpha$.

We have now showed that the theorem holds for any even integer $m$ greater than $2\alpha$. To remove the restriction that $m$ is an even integer we use convexity. We can do so since the scaling function is always convex (\cite[Proposition 2.1(ii)]{GLST2015}). Thus, by applying the following lemma, we conclude that $\tau_{X^*}(q) = q-\alpha$ for any $q\geq q^*$ where $q^*$ is the smallest even integer greater than $2\alpha$.
\end{proof}

\begin{lemma}
Suppose that $\alpha>0$ and $f$ is a convex function such that $f(q)=q-\alpha$ for three values of $q$, namely $q \in \{x,y,z\}$, $x<y<z$. Then the function $f$ must be a straight line segment, i.e.~$f(q)=q-\alpha$ for any $q$ in the interval $[x,z]$.
\end{lemma}

\begin{proof}
Suppose by contradiction that there is $p \in (x,z)$ such that
\begin{equation*}
f(p) < \frac{z-p}{z-x} f(x) + \frac{p-x}{z-x} f(z),
\end{equation*}
so that $f$ is not a straight line. With no loss of generality, focus on the interval $(x,y)$ and let $p \in (x,y)$ so that
\begin{equation}\label{fpineq}
f(p) < \frac{y-p}{y-x} f(x) + \frac{p-x}{y-x} f(y).
\end{equation}
Let $l$ denote the line through points $(p,f(p))$, $(z,f(z))$, i.e.
\begin{equation*}
l(q)=f(z)+ \frac{f(z)-f(p)}{z-p} (q-z).
\end{equation*}
Note that 
\begin{equation*}
l(y) = f(z)+ \left( f(z)-f(p) \right)\frac{y-z}{z-p} = f(z) \frac{y-p}{z-p} + f(p) \frac{z-y}{z-p}.
\end{equation*}
By convexity, we should have $f(y) \leq l(y)$. However, by applying \eqref{fpineq}
\begin{align*}
l(y) &<f(z) \frac{y-p}{z-p} + \frac{z-y}{z-p} \left( \frac{y-p}{y-x} f(x) + \frac{p-x}{y-x} f(y) \right)\\
&=(z-\alpha) \frac{y-p}{z-p} + \frac{z-y}{z-p} \left( \frac{y-p}{y-x} (x-\alpha) + \frac{p-x}{y-x} (y-\alpha) \right)\\
&=y-\alpha=f(y),
\end{align*}
which gives contradiction. Hence, we conclude $f$ is linear.
\end{proof}

We can now apply Theorem \ref{thm:partialsumsupOU:cumulants} to establish the following result for the partial sum supOU process. We do not include the proof since the result is similar to that for the integrated process $X^*$. In fact, the moments and cumulants of $X^+(t)$ and $X^*(t)$ have the same asymptotic behavior as $t\to \infty $, and therefore $\sigma _{X^+}(m)=\sigma _{X^*}(m)$.

\begin{lemma}\label{thm:intermittency-partilsum}
Suppose that the supOU process satisfies the conditions of Proposition \ref{prop:corrasymptotic} and satisfies \eqref{regvarofpi} with some $\alpha>0$, $\kappa_X$ is analytic in a neighborhood of the origin and let $\sigma_{X^+}$ be the cumulant based scaling function \eqref{defsigma} of the partial sum process $\{X^+(t), \, t \geq 0\}$. If $\kappa_X^{(1)}\neq 0$, then
\begin{equation*}
\sigma_{X^+}(1)=1.
\end{equation*}
If $m>\alpha+1$ and $\kappa_X^{(m)}\neq 0$, then
\begin{equation*}
  \sigma_{X^+}(m) = m - \alpha.
\end{equation*}
\end{lemma}

Set $\alpha=2(1-H)$ with $H\in (1/2,1)$ so that $\alpha \in (0,1)$. A special case of Lemma \ref{thm:intermittency-partilsum} was proved in \cite{GLST2015} for the specific situation of the Example \ref{exa:discretesup}. In the notation of Example \ref{exa:discretesup}, the case considered there corresponds to a discrete type superposition $X(t)=\sum_{k=1}^\infty X^{(k)}(t)$ obtained by choosing
\begin{equation*}
\lambda_k=\lambda/k,\ \lambda>0 \ \text{ and } \ p_k = C \zeta(1+2(1-H)) /k^{1+2(1-H)},\ C>0,
\end{equation*}
where $\zeta$ is the Riemann zeta function. In addition, it is assumed that the cumulants of the standard OU type processes $\{X^{(k)}(t)\}$ scale in a specific way. Under these conditions, the cumulants of the centered partial sum process $S(t) = \sum_{i=1}^{\lfloor t \rfloor} \left( X(i) - \E X(i) \right)$ are shown to have the form
\begin{equation*}
  \kappa_S^{(m)} (Nt) = C_m L(N) \lfloor Nt \rfloor^{m-2(1-H)} \left(1+o(1)\right),
\end{equation*}
as $N\to \infty$, where $C_m$ is a positive constant and $L$ a slowly varying function.

Using the same argument as in the proof of Theorem \ref{cor:intermittency-integrated}, we obtain the following result on intermittency of the partial sum process.

\begin{theorem}\label{cor:intermittency-partilsum}
Suppose that for the non-Gaussian supOU process $X$ the assumptions of Lemma \ref{thm:intermittency-integrated} hold with $\alpha >0$, $\kappa_X^{(1)}= 0$ and $\kappa_X^{(2)} \neq 0$. If $\tau_{X^+}$ is the scaling function \eqref{deftau} of $X^+=\{X^+(t), \, t \geq 0\}$, then for every $q\geq q^*$
\begin{equation*}
  \tau_{X^+}(q) = q-\alpha.
\end{equation*}
where $q^*$ is the smallest even integer greater than $2\alpha$. Thus $X^+$ is intermittent.
\end{theorem}

\begin{remark}
In Example \ref{ex:finitesuplimit} (finite superpositions case) and Example \ref{ex:gaussian} (Gaussian case),  we have shown that there is no intermittency. Note that these two cases are clearly not covered in Theorems \ref{cor:intermittency-integrated} and \ref{cor:intermittency-partilsum} where we suppose a non-Gaussian process and regular variation \eqref{regvarofpi} of measure $\pi$.

On the other hand, particular examples of supOU processes satisfying conditions of Theorems \ref{cor:intermittency-integrated} and \ref{cor:intermittency-partilsum} can be obtained by choosing for the marginal distribution any selfdecomposable distribution with zero mean and analytic cumulant function (e.g.~distributions from Examples \ref{ex:X:IG} and \ref{ex:X:NIG}) and by taking the measure $\pi$ that satisfies \eqref{regvarofpi} (e.g.~measures given in Examples \ref{ex:pi:gamma}, \ref{ex:pi:ML} and \ref{ex:pi:ML2}). For any such combination we obtain an intermittent supOU process. Under these conditions, both the integrated and the partial sum process are intermittent. This implies that \eqref{thm:limitdistr} and \eqref{thm:limitmom} cannot both hold. The study of limit theorems for integrated supOU processes and how they relate to the intermittency property will appear in future work.
\end{remark}

\bigskip

\textbf{Acknowledgments:} Nikolai N. Leonenko was supported in part by projects MTM2012-32674 (co-funded by European Regional Development Funds), and MTM2015--71839--P, MINECO, Spain. This research was also supported under Australian Research Council's Discovery Projects funding scheme (project number DP160101366), and under Cardiff Incoming Visiting Fellowship Scheme and International Collaboration Seedcorn Fund.

Murad S.~Taqqu was supported by the NSF grant DMS-1309009 at Boston University.

\bibliographystyle{imsart-number}
\bibliography{References}

\begin{thebibliography}{43}
% BibTex style file: imsart-number.bst, 2013-01-28
% Default style options (sort=1,type=number).
% Used options (sort=1,type=number).

\bibitem{abramowitz1964handbook}
\begin{bbook}[author]
\bauthor{\bsnm{Abramowitz},~\bfnm{Milton}\binits{M.}} \AND
  \bauthor{\bsnm{Stegun},~\bfnm{Irene~A}\binits{I.~A.}}
(\byear{1964}).
\btitle{Handbook of Mathematical Functions: with Formulas, Graphs, and
  Mathematical Tables}.
\bpublisher{Dover}, \baddress{New York}.
\end{bbook}
\endbibitem

\bibitem{barndorff1997processes}
\begin{barticle}[author]
\bauthor{\bsnm{Barndorff-Nielsen},~\bfnm{Ole~E}\binits{O.~E.}}
(\byear{1997}).
\btitle{Processes of normal inverse {G}aussian type}.
\bjournal{Finance and Stochastics}
\bvolume{2}
\bpages{41--68}.
\end{barticle}
\endbibitem

\bibitem{bn2001}
\begin{barticle}[author]
\bauthor{\bsnm{Barndorff-Nielsen},~\bfnm{Ole~E}\binits{O.~E.}}
(\byear{2001}).
\btitle{Superposition of {O}rnstein--{U}hlenbeck type processes}.
\bjournal{Theory of Probability \& Its Applications}
\bvolume{45}
\bpages{175--194}.
\end{barticle}
\endbibitem

\bibitem{barndorff2005burgers}
\begin{barticle}[author]
\bauthor{\bsnm{Barndorff-Nielsen},~\bfnm{Ole~Eiler}\binits{O.~E.}} \AND
  \bauthor{\bsnm{Leonenko},~\bfnm{N.~N.}\binits{N.~N.}}
(\byear{2005}).
\btitle{Burgers' turbulence problem with linear or quadratic external
  potential}.
\bjournal{Journal of Applied Probability}
\bvolume{42}
\bpages{550--565}.
\end{barticle}
\endbibitem

\bibitem{barndorff2005spectral}
\begin{barticle}[author]
\bauthor{\bsnm{Barndorff-Nielsen},~\bfnm{Ole~Eiler}\binits{O.~E.}} \AND
  \bauthor{\bsnm{Leonenko},~\bfnm{N.~N.}\binits{N.~N.}}
(\byear{2005}).
\btitle{Spectral Properties of Superpositions of {O}rnstein-{U}hlenbeck Type
  Processes}.
\bjournal{Methodology and Computing in Applied Probability}
\bvolume{7}
\bpages{335--352}.
\end{barticle}
\endbibitem

\bibitem{barndorff2014assessing}
\begin{barticle}[author]
\bauthor{\bsnm{Barndorff-Nielsen},~\bfnm{Ole~E}\binits{O.~E.}},
  \bauthor{\bsnm{Pakkanen},~\bfnm{Mikko~S}\binits{M.~S.}},
  \bauthor{\bsnm{Schmiegel},~\bfnm{J{\"u}rgen}\binits{J.}} \betal{et~al.}
(\byear{2014}).
\btitle{Assessing relative volatility/intermittency/energy dissipation}.
\bjournal{Electronic Journal of Statistics}
\bvolume{8}
\bpages{1996--2021}.
\end{barticle}
\endbibitem

\bibitem{barndorff2013levy}
\begin{barticle}[author]
\bauthor{\bsnm{Barndorff-Nielsen},~\bfnm{Ole~E}\binits{O.~E.}},
  \bauthor{\bsnm{P{\'e}rez-Abreu},~\bfnm{Victor}\binits{V.}} \AND
  \bauthor{\bsnm{Thorbj{\o}rnsen},~\bfnm{Steen}\binits{S.}}
(\byear{2013}).
\btitle{L{\'e}vy mixing}.
\bjournal{ALEA Latin American Journal of Probability and Mathematical
  Statistics}
\bvolume{10}
\bpages{1013--1062}.
\end{barticle}
\endbibitem

\bibitem{barndorff2009brownian}
\begin{barticle}[author]
\bauthor{\bsnm{Barndorff-Nielsen},~\bfnm{Ole~E}\binits{O.~E.}} \AND
  \bauthor{\bsnm{Schmiegel},~\bfnm{J{\"u}rgen}\binits{J.}}
(\byear{2009}).
\btitle{Brownian semistationary processes and volatility/intermittency}.
\bjournal{Advanced financial modelling}
\bvolume{8}
\bpages{1--26}.
\end{barticle}
\endbibitem

\bibitem{bnshepard2001}
\begin{barticle}[author]
\bauthor{\bsnm{Barndorff-Nielsen},~\bfnm{Ole~E.}\binits{O.~E.}} \AND
  \bauthor{\bsnm{Shephard},~\bfnm{Neil}\binits{N.}}
(\byear{2001}).
\btitle{Non-Gaussian Ornstein–Uhlenbeck-based models and some of their uses
  in financial economics}.
\bjournal{Journal of the Royal Statistical Society: Series B (Statistical
  Methodology)}
\bvolume{63}
\bpages{167--241}.
\bdoi{10.1111/1467-9868.00282}
\end{barticle}
\endbibitem

\bibitem{barndorff2011multivariate}
\begin{barticle}[author]
\bauthor{\bsnm{Barndorff-Nielsen},~\bfnm{Ole~Eiler}\binits{O.~E.}} \AND
  \bauthor{\bsnm{Stelzer},~\bfnm{Robert}\binits{R.}}
(\byear{2011}).
\btitle{Multivariate sup{OU} processes}.
\bjournal{The Annals of Applied Probability}
\bvolume{21}
\bpages{140--182}.
\end{barticle}
\endbibitem

\bibitem{barndorff2013multivariate}
\begin{barticle}[author]
\bauthor{\bsnm{Barndorff-Nielsen},~\bfnm{Ole~Eiler}\binits{O.~E.}} \AND
  \bauthor{\bsnm{Stelzer},~\bfnm{Robert}\binits{R.}}
(\byear{2013}).
\btitle{The multivariate sup{OU} stochastic volatility model}.
\bjournal{Mathematical Finance}
\bvolume{23}
\bpages{275--296}.
\end{barticle}
\endbibitem

\bibitem{barndorff2013stochastic}
\begin{barticle}[author]
\bauthor{\bsnm{Barndorff-Nielsen},~\bfnm{Ole~E}\binits{O.~E.}} \AND
  \bauthor{\bsnm{Veraart},~\bfnm{Almut~ED}\binits{A.~E.}}
(\byear{2013}).
\btitle{Stochastic volatility of volatility and variance risk premia}.
\bjournal{Journal of Financial Econometrics}
\bvolume{11}
\bpages{1--46}.
\end{barticle}
\endbibitem

\bibitem{bingham1989regular}
\begin{bbook}[author]
\bauthor{\bsnm{Bingham},~\bfnm{Nicholas~H}\binits{N.~H.}},
  \bauthor{\bsnm{Goldie},~\bfnm{Charles~M}\binits{C.~M.}} \AND
  \bauthor{\bsnm{Teugels},~\bfnm{Jef~L}\binits{J.~L.}}
(\byear{1989}).
\btitle{Regular Variation}
\bvolume{27}.
\bpublisher{Cambridge University Press}.
\end{bbook}
\endbibitem

\bibitem{carmona1994parabolic}
\begin{bbook}[author]
\bauthor{\bsnm{Carmona},~\bfnm{Ren{\'e}}\binits{R.}} \AND
  \bauthor{\bsnm{Molchanov},~\bfnm{Stanislav~A}\binits{S.~A.}}
(\byear{1994}).
\btitle{Parabolic {A}nderson Problem and Intermittency}
\bvolume{518}.
\bpublisher{Memoirs of the American Mathematical Society}.
\end{bbook}
\endbibitem

\bibitem{chen2015moments}
\begin{barticle}[author]
\bauthor{\bsnm{Chen},~\bfnm{Le}\binits{L.}} \AND
  \bauthor{\bsnm{Dalang},~\bfnm{Robert~C}\binits{R.~C.}}
(\byear{2015}).
\btitle{Moments and growth indices for the nonlinear stochastic heat equation
  with rough initial conditions}.
\bjournal{The Annals of Probability}
\bvolume{43}
\bpages{3006--3051}.
\end{barticle}
\endbibitem

\bibitem{davydov1968convergence}
\begin{barticle}[author]
\bauthor{\bsnm{Davydov},~\bfnm{Yu~A}\binits{Y.~A.}}
(\byear{1968}).
\btitle{Convergence of distributions generated by stationary stochastic
  processes}.
\bjournal{Theory of Probability \& Its Applications}
\bvolume{13}
\bpages{691--696}.
\end{barticle}
\endbibitem

\bibitem{embrechts2002selfsimilar}
\begin{bbook}[author]
\bauthor{\bsnm{Embrechts},~\bfnm{Paul}\binits{P.}} \AND
  \bauthor{\bsnm{Maejima},~\bfnm{Makoto}\binits{M.}}
(\byear{2002}).
\btitle{Selfsimilar Processes}.
\bpublisher{Princeton University Press}.
\end{bbook}
\endbibitem

\bibitem{fasen2007extremes}
\begin{binproceedings}[author]
\bauthor{\bsnm{Fasen},~\bfnm{Vicky}\binits{V.}} \AND
  \bauthor{\bsnm{Kl{\"u}ppelberg},~\bfnm{Claudia}\binits{C.}}
(\byear{2007}).
\btitle{Extremes of sup{OU} Processes}.
In \bbooktitle{Stochastic Analysis and Applications: The Abel Symposium 2005}
\bvolume{2}
\bpages{339--359}.
\bpublisher{Springer Science \& Business Media}.
\end{binproceedings}
\endbibitem

\bibitem{feller1971theory}
\begin{bbook}[author]
\bauthor{\bsnm{Feller},~\bfnm{William}\binits{W.}}
(\byear{1971}).
\btitle{An Introduction to Probability Theory and Its Applications}
\bvolume{2}.
\bpublisher{John Wiley, New York}.
\end{bbook}
\endbibitem

\bibitem{frisch1995turbulence}
\begin{bbook}[author]
\bauthor{\bsnm{Frisch},~\bfnm{Uriel}\binits{U.}}
(\byear{1995}).
\btitle{Turbulence: The Legacy of {A. N.} {K}olmogorov}.
\bpublisher{Cambridge University Press}.
\end{bbook}
\endbibitem

\bibitem{gartner2007geometric}
\begin{barticle}[author]
\bauthor{\bsnm{G{\"a}rtner},~\bfnm{J{\"u}rgen}\binits{J.}},
  \bauthor{\bsnm{K{\"o}nig},~\bfnm{Wolfgang}\binits{W.}} \AND
  \bauthor{\bsnm{Molchanov},~\bfnm{Stanislav~A}\binits{S.~A.}}
(\byear{2007}).
\btitle{Geometric characterization of intermittency in the parabolic {A}nderson
  model}.
\bjournal{The Annals of Probability}
\bvolume{35}
\bpages{439--499}.
\end{barticle}
\endbibitem

\bibitem{GLST2015}
\begin{barticle}[author]
\bauthor{\bsnm{Grahovac},~\bfnm{D.}\binits{D.}},
  \bauthor{\bsnm{Leonenko},~\bfnm{N.~N.}\binits{N.~N.}},
  \bauthor{\bsnm{Sikorskii},~\bfnm{A.}\binits{A.}} \AND \bauthor{\bsnm{Te{\v
  s}njak},~\bfnm{I.}\binits{I.}}
(\byear{2016}).
\btitle{Intermittency of Superpositions of {O}rnstein-{U}hlenbeck Type
  Processes}.
\bjournal{Journal of Statistical Physics, to appear}.
\bnote{doi:10.1007/s10955-016-1616-7}.
\end{barticle}
\endbibitem

\bibitem{griffin2010bayesian}
\begin{barticle}[author]
\bauthor{\bsnm{Griffin},~\bfnm{Jim~E}\binits{J.~E.}} \AND
  \bauthor{\bsnm{Steel},~\bfnm{Mark~FJ}\binits{M.~F.}}
(\byear{2010}).
\btitle{Bayesian inference with stochastic volatility models using continuous
  superpositions of non-{G}aussian {O}rnstein--{U}hlenbeck processes}.
\bjournal{Computational Statistics \& Data Analysis}
\bvolume{54}
\bpages{2594--2608}.
\end{barticle}
\endbibitem

\bibitem{gupta2009cumulants}
\begin{barticle}[author]
\bauthor{\bsnm{Gupta},~\bfnm{Arjun~K}\binits{A.~K.}},
  \bauthor{\bsnm{Shanbhag},~\bfnm{Damodar~N}\binits{D.~N.}},
  \bauthor{\bsnm{Nguyen},~\bfnm{Truc~T}\binits{T.~T.}} \AND
  \bauthor{\bsnm{Chen},~\bfnm{JT}\binits{J.}}
(\byear{2009}).
\btitle{Cumulants of infinitely divisible distributions}.
\bjournal{Random Operators and Stochastic Equations}
\bvolume{17}
\bpages{103--124}.
\end{barticle}
\endbibitem

\bibitem{heyde2005student}
\begin{barticle}[author]
\bauthor{\bsnm{Heyde},~\bfnm{CC}\binits{C.}} \AND
  \bauthor{\bsnm{Leonenko},~\bfnm{Nikolai~N}\binits{N.~N.}}
(\byear{2005}).
\btitle{Student processes}.
\bjournal{Advances in Applied Probability}
\bvolume{37}
\bpages{342--365}.
\end{barticle}
\endbibitem

\bibitem{LeonenkoIvanov}
\begin{bbook}[author]
\bauthor{\bsnm{Ivanov},~\bfnm{A.~A.}\binits{A.~A.}} \AND
  \bauthor{\bsnm{Leonenko},~\bfnm{N.~N.}\binits{N.~N.}}
(\byear{1989}).
\btitle{Statistical Analysis of Random Fields}.
\bpublisher{Springer Netherlands}.
\end{bbook}
\endbibitem

\bibitem{jurek2001remarks}
\begin{barticle}[author]
\bauthor{\bsnm{Jurek},~\bfnm{Zbigniew~J}\binits{Z.~J.}}
(\byear{2001}).
\btitle{Remarks on the selfdecomposability and new examples}.
\bjournal{Demonstratio Mathematica}
\bvolume{34}
\bpages{241--250}.
\end{barticle}
\endbibitem

\bibitem{khoshnevisan2014analysis}
\begin{bbook}[author]
\bauthor{\bsnm{Khoshnevisan},~\bfnm{Davar}\binits{D.}}
(\byear{2014}).
\btitle{Analysis of Stochastic Partial Differential Equations}
\bvolume{119}.
\bpublisher{American Mathematical Society}.
\end{bbook}
\endbibitem

\bibitem{leonenko2005convergence}
\begin{barticle}[author]
\bauthor{\bsnm{Leonenko},~\bfnm{Nikolai~N.}\binits{N.~N.}} \AND
  \bauthor{\bsnm{Taufer},~\bfnm{Emanuele}\binits{E.}}
(\byear{2005}).
\btitle{Convergence of integrated superpositions of {O}rnstein-{U}hlenbeck
  processes to fractional {B}rownian motion}.
\bjournal{Stochastics: An International Journal of Probability and Stochastics
  Processes}
\bvolume{77}
\bpages{477--499}.
\end{barticle}
\endbibitem

\bibitem{lukacs1970characteristic}
\begin{bbook}[author]
\bauthor{\bsnm{Lukacs},~\bfnm{E.}\binits{E.}}
(\byear{1970}).
\btitle{Characteristic Functions}.
\bpublisher{Hafner Publishing Company}.
\end{bbook}
\endbibitem

\bibitem{masuda2004}
\begin{barticle}[author]
\bauthor{\bsnm{Masuda},~\bfnm{Hiroki}\binits{H.}}
(\byear{2004}).
\btitle{On multidimensional {O}rnstein-{U}hlenbeck processes driven by a
  general {L}{\'e}vy process}.
\bjournal{Bernoulli}
\bvolume{10}
\bpages{97--120}.
\end{barticle}
\endbibitem

\bibitem{molchanov1991ideas}
\begin{barticle}[author]
\bauthor{\bsnm{Molchanov},~\bfnm{Stanislav~A}\binits{S.~A.}}
(\byear{1991}).
\btitle{Ideas in the theory of random media}.
\bjournal{Acta Applicandae Mathematica}
\bvolume{22}
\bpages{139--282}.
\end{barticle}
\endbibitem

\bibitem{moser2011tail}
\begin{barticle}[author]
\bauthor{\bsnm{Moser},~\bfnm{Martin}\binits{M.}} \AND
  \bauthor{\bsnm{Stelzer},~\bfnm{Robert}\binits{R.}}
(\byear{2011}).
\btitle{Tail behavior of multivariate {L}{\'e}vy-driven mixed moving average
  processes and {supOU} stochastic volatility models}.
\bjournal{Advances in Applied Probability}
\bvolume{43}
\bpages{1109--1135}.
\end{barticle}
\endbibitem

\bibitem{oodaira1972functional}
\begin{barticle}[author]
\bauthor{\bsnm{Oodaira},~\bfnm{Hiroshi}\binits{H.}} \AND
  \bauthor{\bsnm{Yoshihara},~\bfnm{Kenichi}\binits{K.}}
(\byear{1972}).
\btitle{Functional central limit theorems for strictly stationary processes
  satisfying the strong mixing condition}.
\bjournal{Kodai Mathematical Seminar Reports}
\bvolume{24}
\bpages{259--269}.
\end{barticle}
\endbibitem

\bibitem{peccati2011wiener}
\begin{bbook}[author]
\bauthor{\bsnm{Peccati},~\bfnm{Giovanni}\binits{G.}} \AND
  \bauthor{\bsnm{Taqqu},~\bfnm{Murad~S}\binits{M.~S.}}
(\byear{2011}).
\btitle{Wiener Chaos: Moments, Cumulants and Diagrams: A survey with computer
  implementation}.
\bpublisher{Springer Science \& Business Media}.
\end{bbook}
\endbibitem

\bibitem{podolskij2015}
\begin{binbook}[author]
\bauthor{\bsnm{Podolskij},~\bfnm{Mark}\binits{M.}}
(\byear{2015}).
\btitle{Ambit Fields: Survey and New Challenges}
In \bbooktitle{XI Symposium on Probability and Stochastic Processes: CIMAT,
  Mexico, November 18-22, 2013}
\bpages{241--279}.
\bpublisher{Springer International Publishing}, \baddress{Cham}.
\bdoi{10.1007/978-3-319-13984-5_12}
\end{binbook}
\endbibitem

\bibitem{rajput1989spectral}
\begin{barticle}[author]
\bauthor{\bsnm{Rajput},~\bfnm{Balram~S}\binits{B.~S.}} \AND
  \bauthor{\bsnm{Rosinski},~\bfnm{Jan}\binits{J.}}
(\byear{1989}).
\btitle{Spectral representations of infinitely divisible processes}.
\bjournal{Probability Theory and Related Fields}
\bvolume{82}
\bpages{451--487}.
\end{barticle}
\endbibitem

\bibitem{stelzer2015moment}
\begin{barticle}[author]
\bauthor{\bsnm{Stelzer},~\bfnm{Robert}\binits{R.}},
  \bauthor{\bsnm{Tosstorff},~\bfnm{Thomas}\binits{T.}} \AND
  \bauthor{\bsnm{Wittlinger},~\bfnm{Marc}\binits{M.}}
(\byear{2015}).
\btitle{Moment based estimation of {supOU} processes and a related stochastic
  volatility model}.
\bjournal{Statistics \& Risk Modeling}
\bvolume{32}
\bpages{1--24}.
\end{barticle}
\endbibitem

\bibitem{Stelzer2015}
\begin{binproceedings}[author]
\bauthor{\bsnm{Stelzer},~\bfnm{Robert}\binits{R.}} \AND
  \bauthor{\bsnm{Zavi{\v{s}}in},~\bfnm{Jovana}\binits{J.}}
(\byear{2015}).
\btitle{Derivative Pricing under the Possibility of Long Memory in the sup{OU}
  Stochastic Volatility Model}.
In \bbooktitle{Innovations in Quantitative Risk Management: TU M{\"u}nchen,
  September 2013}
\bvolume{99}
\bpages{75--92}.
\bpublisher{Springer}.
\end{binproceedings}
\endbibitem

\bibitem{stoyanov1997counterexamples}
\begin{bbook}[author]
\bauthor{\bsnm{Stoyanov},~\bfnm{J.}\binits{J.}}
(\byear{1997}).
\btitle{Counterexamples in Probability}.
\bpublisher{John Wiley \& Sons}, \baddress{New York}.
\end{bbook}
\endbibitem

\bibitem{Taqqu1975}
\begin{barticle}[author]
\bauthor{\bsnm{Taqqu},~\bfnm{Murad~S}\binits{M.~S.}}
(\byear{1975}).
\btitle{Weak convergence to fractional {B}rownian motion and to the
  {R}osenblatt process}.
\bjournal{Zeitschrift f{\"u}r Wahrscheinlichkeitstheorie und Verwandte Gebiete}
\bvolume{31}
\bpages{287--302}.
\end{barticle}
\endbibitem

\bibitem{yokoyama1980moment}
\begin{barticle}[author]
\bauthor{\bsnm{Yokoyama},~\bfnm{Ryozo}\binits{R.}}
(\byear{1980}).
\btitle{Moment bounds for stationary mixing sequences}.
\bjournal{Zeitschrift f{\"u}r Wahrscheinlichkeitstheorie und Verwandte Gebiete}
\bvolume{52}
\bpages{45--57}.
\end{barticle}
\endbibitem

\bibitem{zel1987intermittency}
\begin{barticle}[author]
\bauthor{\bsnm{Zel'dovich},~\bfnm{YA~B}\binits{Y.~B.}},
  \bauthor{\bsnm{Molchanov},~\bfnm{S~A}\binits{S.~A.}},
  \bauthor{\bsnm{Ruzma{\u\i}kin},~\bfnm{AA}\binits{A.}} \AND
  \bauthor{\bsnm{Sokolov},~\bfnm{Dmitrii~D}\binits{D.~D.}}
(\byear{1987}).
\btitle{Intermittency in random media}.
\bjournal{Soviet Physics Uspekhi}
\bvolume{30}
\bpages{353}.
\end{barticle}
\endbibitem

\end{thebibliography}

\end{document}